\documentclass[a4paper,11.5pt]{article}
\usepackage{tikz}
\usetikzlibrary{calc}
\usepackage{amssymb}
\usetikzlibrary{quotes,angles}
\usetikzlibrary{plotmarks}
\usetikzlibrary{arrows,shapes,positioning}
\usetikzlibrary{decorations.markings}
\usepackage{geometry}
\tikzstyle arrowstyle=[scale=2]
\tikzstyle directed=[postaction={decorate,decoration={markings,
    mark=at position .65 with {\arrow[arrowstyle]{stealth}}}}]
\tikzstyle reverse directed=[postaction={decorate,decoration={markings,
    mark=at position .65 with {\arrowreversed[arrowstyle]{stealth};}}}]
    \tikzstyle left directed=[postaction={decorate,decoration={markings,
    mark=at position -.62 with {\arrow[arrowstyle]{stealth}}}}]

\tikzstyle left reverse directed=[postaction={decorate,decoration={markings,
    mark=at position -.62 with {\arrowreversed[arrowstyle]{stealth};}}}]
\usepackage{indentfirst}
\usepackage[plainpages=false]{hyperref}
\usepackage{amsfonts,latexsym,rawfonts,amsmath,amssymb,amsthm,mathrsfs}
\usepackage{amsmath,amssymb,amsfonts,latexsym,lscape,rawfonts}

\usepackage[all]{xy}
\usepackage{eufrak}
\usepackage{makeidx}         
\usepackage{graphicx,psfrag}
\usepackage{array,tabularx}

\usepackage{setspace}

\newtheorem{thm}{Theorem}[section]
\newtheorem{cor}[thm]{Corollary}
\newtheorem{lem}[thm]{Lemma}
\newtheorem{clm}[thm]{Claim}
\newtheorem{prop}[thm]{Proposition}

\theoremstyle{remark}
\newtheorem{rmk}[thm]{Remark}

\theoremstyle{definition}
\newtheorem{Def}[thm]{Definition}                                        %

\def \R {\mathbb R}

\title{An elliptic theory of indicial weights  and  applications to non-linear geometry problems}

{\small{\author{Yuanqi Wang\thanks{Department of Mathematics, Stony Brook University, NY, USA.\ ywang@scgp.stonybrook.edu.}}}

\date{\vspace{-5ex}}

\begin{document}
\maketitle
\begin{abstract} Given an elliptic operator  $P$     on a non-compact manifold (with proper asymptotic conditions), there is a discrete set of numbers called indicial roots.  It's known that $P$  is Fredholm between  weighted Sobolev spaces  if and only if the weight is not indicial. We show that an elliptic theory exists even when the weight is indicial. We also discuss some simple applications to Yang-Mills theory and  minimal surfaces.  
\end{abstract}

\section{Introduction}
\subsection{The theory}
The elliptic theories based on weighted Sobolev (Schauder) spaces usually concern a  discrete set of real numbers. If a number is  in the set, we say that it is indicial (or is an indicial root).  A classical fact says that  on a non-compact complete manifold,    an elliptic operator (with proper asymptotic conditions) is Fredholm between  weighted Sobolev spaces  if and only if the weight is not  indicial. For earlier pioneering work, please see   \cite{LockhartActa}, \cite{Lockhart}, and  \cite{MelroseMendoza}. For  more recent work, please see \cite{Mazzeo}.

 Following  elementary ideas, we show that there is an elliptic theory even if the weight is indicial: first, we add   polynomial weights  \{compare (\ref{equ weighted L2 norm in Def Graph norm}) to \cite[(1.3)]{Lockhart}\} to refine the space; second, we consider graph norms with respect to the model operator  [see (\ref{eqnarray the theories in Thm Fredholm})].

In this note we only consider first and second-order operators modelled on the following. 
\begin{Def}\label{Def TID operators}  Let $Y$ be a $(n-1)-$dimensional  Riemannian manifold without boundary (which does not have to be connected). Let $E,\ F$ be smooth vector-bundles over $Y$ equipped with  smooth Hermitian metrics. Given arbitrary bundle isomorphisms $\sigma_{1}: E\rightarrow F,\ \sigma_{2}: E\rightarrow E$, we say that an   operator $P^{0}$ is TID (translation-invariant and   diagonal) if  
\begin{equation}\label{equ Def P0 TID}P^{0}=\sigma_{1}(-B_{P^0}-a_{1}\frac{\partial}{\partial t}+a_{2}\frac{\partial^{2}}{\partial t^{2}})\sigma_{2}\ \textrm{and the following holds}.\end{equation} 
\begin{itemize}
\item $a_{2}=0$ or $1$. $a_{1}=-1$ when $a_{2}=0$ (always achievable by normalization).
\item  When $a_{2}=0$, $B_{P^{0}}$ is a first-order self-adjoint elliptic differential operator $C^{\infty}(Y,E)\rightarrow C^{\infty}(Y,E)$.  When $a_{2}=1$, $B_{P^{0}}$ is second-order, simple,   elliptic, and self-adjoint $C^{\infty}(Y,E)\rightarrow C^{\infty}(Y,E)$  (see Definition \ref{Def Simple 2nd order operators}).
  \end{itemize} 
 \end{Def}
\begin{rmk}\label{Rmk eigen-basis exists} For any TID operator $P^{0}$, $Spec B_{P^{0}}$ is real and discrete. Moreover, there is a complete eigen-basis of $B_{P^{0}}$. 
\end{rmk}
\begin{Def} \label{Def Indicial weights} Let  $P^{0}$ be TID, and $(\beta,\Lambda)$ be a pair of real numbers such that $\Lambda\in Spec(B_{P^{0}})$. When  $P^{0}$ is first-order, we say that $(\beta,\Lambda)$ is  $P^{0}-$indicial if $\beta=\Lambda$. When $P^{0}$ is second-order, we say that $(\beta,\Lambda)$ is  $P^{0}-$indicial if
\begin{enumerate} 
\item  $\beta\neq \frac{a_{1}}{2}$ and $\Lambda-\beta^{2}+a_{1}\beta=0$, or if
\item $\Lambda\leq -\frac{a_{1}^{2}}{4}$ and $\beta= \frac{a_{1}}{2}$.
\end{enumerate}  
In the second case above, we say  that $(\beta,\Lambda)$  is  $P^{0}-$super indicial. We say that $\beta$ is $P^{0}-$indicial (super indicial) if there is a $\Lambda\in SpecB_{P^{0}}$ such that  $(\beta,\Lambda)$ is $P^{0}-$indicial (super indicial). This is consistent with the "$\mathfrak{D}_{A}$" in \cite[page 417]{Lockhart}, translated to our setting.
\end{Def}
 Let $N$ be a  complete Riemannian manifold with finite many cylindrical ends, we  consider asymptotically  TID operators $P: C^{\infty}(N,E)\rightarrow C^{\infty}(N,F)$.  This class  should include most of the Dirac and Laplace-type operators in geometry.  In the setting as Theorem \ref{Thm Fredholm in the cyl setting},
\begin{eqnarray}\label{eqnarray the theories in Thm Fredholm}
& & P:\ \widehat{W}^{k+m_{0},p}_{-\beta,\gamma,b-1}(N,E)\longrightarrow W^{k,p}_{-\beta,\gamma,b}(N,F) \ \ \ \ \ \ \ \ \ \ \ \ \  (\square|^{Sobolev,p}_{-\beta,\gamma,b})\ (\textrm{see  Definition}\ \ref{Def graph norm},\ \ref{Def Global graph norm and the beta vector}\nonumber\\& &
P:\ \widehat{C}^{k+m_{0},\alpha}_{-\beta,\gamma,b-1}(N,E)\longrightarrow C^{k,\alpha}_{-\beta,\gamma,b}(N,F) \ \ \ \ \ \ \ \ \ \ \ \ \ \ \   (\square|^{Schauder}_{-\beta,\gamma,b})\ \ \textrm{ for the norms}).
\end{eqnarray}
are bounded operators. Moreover, when $\beta$ is not indicial and  $\gamma=0$, as subspaces of $L^{2}_{loc}$,
\begin{equation}\begin{array}{c} \widehat{W}^{k+m_{0},p}_{-\beta,\gamma,b-1}(N,E)=W^{k+m_{0},p}_{-\beta}(N,E),\  \widehat{C}^{k+m_{0},\alpha}_{-\beta,\gamma,b-1}(N,E)=C^{k+m_{0},\alpha}_{-\beta}(N,E), \\ W^{k,p}_{-\beta,\gamma,b}(N,F)=W^{k,p}_{-\beta}(N,F),\  C^{k,\alpha}_{-\beta,\gamma,b}(N,F)=C^{k,\alpha}_{-\beta}(N,F). \end{array}\end{equation}
Thus  our theory  generalizes the  one in \cite{Lockhart} \{for  first and second-order operators, c.f \cite[(1.3)]{Lockhart}\}.  Assuming the weights are the same on the ends, our main result states as follows. 
\begin{thm}\label{Thm Fredholm in the cyl setting} Suppose $P$  is a $\beta-ATID$ elliptic  operator (see Definition \ref{Def ATID}), and $\beta$ is not $P^{0}-$super indicial. Then for any $0\leq k\leq k_{0}-2$, $\alpha\in (0,1)$, $p\geq 2$, 
\begin{itemize}
\item $(\square|^{Sobolev,p}_{-\beta,\gamma, b})$ is Fredholm if  $b\neq 1-\frac{1}{p}$ or $\beta\ \textrm{is not}\ P^{0}-\textrm{indicial}$;
\item $(\square|^{Schauder}_{-\beta,\gamma,b})$ is Fredholm if  $b\neq 1$ or $\beta\ \textrm{is not}\ P^{0}-\textrm{indicial}$.
\end{itemize} 
\end{thm}
\begin{rmk} The super-indicial roots are essentially different from the ordinary ones. Fortunately, they don't exist for first-order operators, and  they barely appear on second-order operators. For example, any super-indicial root in the setting of Corollary \ref{Cor Minimal surfaces} must be positive, but we only need $\beta$ to be non-positive  therein.  
\end{rmk}
\begin{rmk} Theorem \ref{Thm general index} and Proposition \ref{Prop regularity of Harmonic sections} give   reasonably general  index formulas for first-order operators (see Remark \ref{rmk all the index can be computed}). As a by-product, we prove an obvious identity (Proposition \ref{prop eta}) on the eta-invariant defined in \cite{APS}. It can also be proved by the Fredholm theory in \cite{APS}. However, the author is not able to find Proposition \ref{prop eta} in the literature.  
\end{rmk}
\begin{rmk}\label{Rmk ker coker index do not change}  Though the indicial roots do not prevent Fredholmnness, the index still changes when $\beta$ goes across any of them (c.f.  \cite[Last 5 lines in Page 433]{Lockhart}).   
\end{rmk}
Our theory still works   when the weights are not the same on the ends (see Theorem \ref{Thm Fredholm  different weight on different ends}).

Computations indicate that the our local inverses (Theorem \ref{Thm local invertibility}) are different from those of Lockhart-McOwen \cite[(2.3)]{Lockhart} (by Fourier-transform in the $t-$direction). When $k<0$, our local inverses do not work for the   $W^{k,p}$ ($C^{k,\alpha}$) theories. 
\begin{rmk}\label{Rmk not Fredholm} Assuming that $P$ is translation invariant on each end, Theorem \ref{Thm Fredholm in the cyl setting}, \ref{Thm Fredholm  different weight on different ends} are still true with "if" replaced by "if and only if" (see the proof in the Appendix). However,   only assuming $\beta-$ATID, when $b=1$, we don't know whether
 $(\square|^{Schauder}_{-\beta,\gamma,b})$ is not Fredholm. The same doubt applies to the Sobolev theory. 
\end{rmk}

By simple conformal changes as in \cite[Section 9]{Lockhart},  Theorem \ref{Thm Fredholm in the cyl setting} is equivalent to a theory in the conic setting (and hopefully the asymptotic conic setting). Please also see \cite{Myself2016bigpaper} and 
the discussion above Lemma \ref{lem Coulomb gauge}.

   Under stronger asymptotic conditions than Definition \ref{Def ATID},  we  have a theory for super-indicial roots (on second-order operators), and a theory for  powers of Laplace-type operators i.e. $\Delta^{m},\ (m\geq 2)$. The  higher-order operators  include (linearisation of)  the extremal metric operator in \cite{Calabi}, and the conformal co-variant operators in \cite{ChangYang}. We will address these in the future when geometric motivation arises.

 Amrouch-Girault-Giroire \cite{Amrouch} also use Sobolev-spaces with polynomial weights to study Laplace equations on  domains. It's possible that our theory is essentially similar to theirs.
 
 \subsection{Simple applications}
 Geometric objects with isolated conic singularities usually converge to their tangent cones polynomially (see \cite{LSimon}). Let $r$ be the distance to the singular point,  and $t=-\log r$ be the cylindrical coordinate. Our work  implies a general phenomenon:  the rate of convergence to the tangent cone is either exponential or not faster than $\frac{1}{t}$ (or $\frac{1}{-\log r}$). 
  
  We first do minimal sub-manifolds. In the cylindrical setting, we say that a minimal graph sub-manifold is asymptotic to a cone at a certain rate, if the section "$u$" in (\ref{equ 1 in proof Cor Minimal surface}) converges to $0$  at the rate (see Definition \ref{Def convergence}). 
\begin{cor}\label{Cor Minimal surfaces}  Suppose $\Sigma$ is a $n-$dimensional closed minimal sub-manifold in $\mathbb{S}^{N},\ n\geq 1$. Let $\underline{0}$ be the negative number    in Definition \ref{Def Q+ and Q-} with respect to the $L_{\Sigma}$ in (\ref{equ 1 in proof Cor Minimal surface}). Then there is a $\delta_{0}$ depending on  $\Sigma$ with the following property. 

 Suppose $\widehat{\Sigma}$ is a (locally defined) embedded minimal sub-manifold in $\R^{N+1}$ with isolated cone singularity at $O$. Suppose $\widehat{\Sigma}$ is a graph over $Cone(\Sigma)$, and in the cylindrical setting, it converges to  $Cone(\Sigma)$ at least at the rate  $\frac{\delta_{0}}{t}$ (see Definition \ref{Def convergence}). Then $\widehat{\Sigma}$ converges to $Cone(\Sigma)$ exponentially at the rate $O(e^{-|\underline{0}|t})$. 
\end{cor}

\begin{rmk}By  Definition \ref{Def ATID} and Remark \ref{Rmk Dependence of delta 0},  we can not make $\delta_{0}$ small by scaling. Adam-Simon \cite{AdamSimon} showed that there are singular minimal sub-manifolds converging to a cone  at a rate comparable to  $(-\log r)^{-1}$. This suggests  that in general, the assumption on the rate in Corollary \ref{Cor Minimal surfaces} can not be weaken.   
\end{rmk}

Similar results hold for Yang-Mills connections as well.
\begin{cor}\label{Cor regularity Yang Mills} Let $n\geq 5$.  Suppose $g$ is a smooth metric  on  $B^{n}_{O}(R)$  and $g(O)=g_{E}$ (the Euclidean metric). Suppose $A_{O}$ is a $U(m)$ or $SO(m)$  Yang-Mills connection on (the unit round) $\mathbb{S}^{n-1}$. Let $\underline{0}$ be the negative number    in Definition \ref{Def Q+ and Q-} with respect to the $B$ in (\ref{equ 2 proof Yang Mills Cor}). Then there is a $\delta_{0}>0$ depending on  $A_{O}$
with the following properties. 

 Suppose  $A$ is a smooth Yang-Mills connection on $B_{O}(R)\setminus O$. In the cylindrical setting as Section \ref{section Yang Mills application}, suppose $A$ converges to  $Cone(A_{O})$ at least at the rate $\frac{\delta_{0}}{t}$ (see Definition \ref{Def convergence}).

$I:$ Suppose  $A$  is in Coulomb gauge relative to $A_{O}$ (with respect to $g$ or the Euclidean metric). Then $A$  converges to $Cone(A_{O})$ exponentially at the following rate.
\begin{equation}\label{equ ROC in cor improving rate G2 monopoles}\left\{\begin{array}{cc}
O(e^{-|\underline{0}|t}) &\ \textrm{when}\ |\underline{0}|<1,\\
O(e^{-t})&\ \textrm{when}\ |\underline{0}|>1,\\
O(te^{-|\underline{0}|t})\  &\ \textrm{when}\ |\underline{0}|=1. 
\end{array}\right. \end{equation}
$II:$ When $A_{O}$ is irreducible,   there exists a   gauge  $s$  such that $s(A)$  converges to $Cone(A_{O})$ exponentially  as  (\ref{equ ROC in cor improving rate G2 monopoles}).  \end{cor}
\begin{rmk}By (\ref{equ 1 proof Yang Mills Cor}), when $n=4$, the weight $0$ is super-indicial  unless $B$ is positively definite. 
\end{rmk}
\textbf{Organization of this note:} the norms  can be found in Section \ref{section preparation}. We construct the local inverses in Section \ref{section Local inverse between weighted Sobolev graph-space}. In Section \ref{section Regularity and Fredholm}, we study regularity of harmonic sections, and complete the proof for Theorem \ref{Thm Fredholm in the cyl setting}. We give the index formula (for first-order operators) and study the eta-invariant in Section \ref{section index}. We prove Corollary \ref{Cor Minimal surfaces},  \ref{Cor regularity Yang Mills} in Section \ref{section Applications}.\\

 \textbf{Acknowledgement:} The author would like to thank Professor Simon Donaldson, Thomas Walpuski, and Lorenzo Foscolo for  helpful discussions.

\section{Preparation \label{section preparation}}\begin{Def}\label{Def Simple 2nd order operators}In the setting of Definition \ref{Def TID operators}, we say that an operator $H:\ C^{\infty}(Y,E)\rightarrow C^{\infty}(Y,E)$ is  admissible, if there is a linear first-order  differential operator  $H_{0}$, and sections $\widehat{\Gamma}_{k},\ \Gamma_{k}\in C^{\infty}(Y,E)$  such that 
\begin{equation}\label{equ Def admissible 1st order operators} H\xi=H_{0}\xi+\Sigma_{k=1}^{k_{0}}\widehat{\Gamma}_{k}\int_{Y}<\xi,\Gamma_{k}>dV.
\end{equation}
We say that a second-order   operator $B:\ C^{\infty}(Y,E)\rightarrow C^{\infty}(Y,E)$ is  simple, if there is a smooth connection $A_{0}$ on $E$ and a smooth metric $\widehat{g}$ on $Y$ (which does not need to be the given one),  such that $B-\nabla^{\star_{\widehat{g}}}_{A_{0}}\nabla_{A_{0}}$ is admissible. 
\end{Def}

\begin{Def}\label{Def Strips}(Strips) Let $kS_{m}=(m-k,m+k)$, and $\{\xi_{m}\}^{\infty}_{m=3}$ be a partition of unity of $Cyl_{1}$ subordinate to the  cover $\{2S_{m}\}^{\infty}_{m=3}$ i.e. $\xi_{m}$ is supported in $2S_{m}$ and is $\equiv 1$ in $S_{m}$. 
 \end{Def}
 \begin{Def}\label{Def graph norm}Let $Cyl_{t_{0}}$ denote $Y\times (t_{0},\infty),\ t_{0}\geq 0.1$,  we define the  $L^{2}_{-\beta,b}[Cyl_{t_{0}}]-$ space of sections to the underlying bundle by the norm  \begin{equation}\label{equ weighted L2 norm in Def Graph norm}|\xi|_{L^{p}_{-\beta,b}(Cyl_{t_{0}})}\triangleq |e^{-\beta t}t^{b} \xi|_{L^{p}(Cyl_{t_{0}})}=(\int_{Cyl_{t_{0}}}|e^{-\beta t}t^{b} \xi|^{p})^{\frac{1}{p}}.
 \end{equation} 
We define   $|\xi|_{W^{k,p}_{-\beta,b}(Cyl_{t_{0}})}\triangleq \Sigma_{j=0}^{k}|\nabla^{j}_{A_{0}}\xi|_{L^{p}_{-\beta,b}(Cyl_{t_{0}})}.$ For the Schauder theory, we define $|\xi|_{C^{k,\alpha}_{-\beta,b}(\overline{Cyl}_{t_{0}})}\triangleq \sup_{m\geq t_{0}+1}m^{b}e^{-\beta m}|\xi|_{C^{k,\alpha}(S_{m})}$. 

 Let  $\xi^{\parallel_\beta}$ denote the projection of $\xi$ onto $Ker\{B_{P^{0}}-\beta Id\}$ (for all $t$), and $\xi^{\perp_\beta}=\xi-\xi^{\parallel_\beta}$ be the perpendicular vector. When $P^{0}$ is first-order, we define 
\begin{eqnarray*}|\sigma_{2}^{-1}\xi|_{\widehat{W}^{k,p,P^{0}}_{-\beta,\gamma,b-1}(Cyl_{t_{0}})}& \triangleq &  |\xi^{\perp_{\beta}} |_{W^{k,p}_{-\beta,\gamma}(Cyl_{t_{0}})}+|\xi^{\parallel_{\beta}}|_{W^{k,p}_{-\beta,b-1}(Cyl_{t_{0}})}+|\frac{\partial \xi^{\parallel_{\beta}}}{\partial t}-\beta\xi^{\parallel_{\beta}}|_{W^{k-1,p}_{-\beta,b}(Cyl_{t_{0}})}
\\ |\sigma_{2}^{-1}\xi|_{W^{k,p,P^{0}}_{-\beta,\gamma,b}(Cyl_{t_{0}})}&\triangleq& |\xi^{\perp_{\beta}} |_{W^{k,p}_{-\beta,\gamma}(Cyl_{t_{0}})}+|\xi^{\parallel_{\beta}}|_{W^{k,p}_{-\beta,b}(Cyl_{t_{0}})}\label{equ Def elliptic graph norm first order}
\\\label{equ Def elliptic graph Schauder norm first order}|\sigma_{2}^{-1}\xi|_{\widehat{C}^{k,\alpha,P^{0}}_{-\beta,\gamma,b-1}(\overline{Cyl}_{t_{0}})}&\triangleq & |\xi^{\perp_{\beta}} |_{C^{k,\alpha}_{-\beta,\gamma}(\overline{Cyl}_{t_{0}})}+|\xi^{\parallel_{\beta}}|_{C^{k,\alpha}_{-\beta,b-1}(\overline{Cyl}_{t_{0}})}+|\frac{\partial \xi^{\parallel_{\beta}}}{\partial t}-\beta\xi^{\parallel_{\beta}}|_{C^{k-1,\alpha}_{-\beta,b}(\overline{Cyl}_{t_{0}})}.\\
|\sigma_{2}^{-1}\xi|_{C^{k,\alpha,P^{0}}_{-\beta,\gamma,b}(\overline{Cyl}_{t_{0}})}&\triangleq& |\xi^{\perp_{\beta}} |_{C^{k,\alpha}_{-\beta,\gamma}(\overline{Cyl}_{t_{0}})}+|\xi^{\parallel_{\beta}}|_{C^{k,\alpha}_{-\beta,b}(\overline{Cyl}_{t_{0}})}.\end{eqnarray*}
When $P^{0}$ is second-order elliptic, let $\Lambda_{\beta}=\beta^{2}-a_{1}\beta$, we define 
\begin{eqnarray*}\label{equ Def elliptic graph norm second order}\nonumber |\sigma_{2}^{-1}\xi|_{\widehat{W}^{k,p,P^{0}}_{-\beta,\gamma,b-1}(Cyl_{t_{0}})}&\triangleq& |\xi^{\perp_{\Lambda_\beta}} |_{W^{k,p}_{-\beta,\gamma}(Cyl_{t_{0}})}+|\xi^{\parallel_{\Lambda_\beta}}|_{W^{k,p}_{-\beta,b-1}(Cyl_{t_{0}})}+|\frac{\partial \xi^{\parallel_{\Lambda_\beta}}}{\partial t}-\beta\xi^{\parallel_{\Lambda_\beta}}|_{W^{k-1,p}_{-\beta,b}(Cyl_{t_{0}})}\\& & +|\frac{\partial^{2} \xi^{\parallel_{\Lambda_\beta}}}{\partial t^{2}}-\beta^{2}\xi^{^{\parallel_{\Lambda_\beta}}}|_{W^{k-2,p}_{-\beta,b}(Cyl_{t_{0}})}\\
|\sigma_{2}^{-1}\xi|_{W^{k,p,P^{0}}_{-\beta,\gamma,b}(Cyl_{t_{0}})}&\triangleq & |\xi^{\perp_{\Lambda_\beta}} |_{W^{k,p}_{-\beta,\gamma}(Cyl_{t_{0}})}+|\xi^{\parallel_{\Lambda_\beta}}|_{W^{k,p}_{-\beta,b}(Cyl_{t_{0}})}\label{equ Def elliptic graph Schauder norm second order}\\ |\sigma_{2}^{-1}\xi|_{\widehat{C}^{k,\alpha,P^{0}}_{-\beta,\gamma,b-1}(\overline{Cyl}_{t_{0}})}&\triangleq & |\xi^{\perp_{\Lambda_\beta}} |_{C^{k,\alpha}_{-\beta,\gamma}(\overline{Cyl}_{t_{0}})}+|\xi^{\parallel_{\Lambda_\beta}}|_{C^{k,\alpha}_{-\beta,b-1}(\overline{Cyl}_{t_{0}})}+|\frac{\partial \xi^{\parallel_{\Lambda_\beta}}}{\partial t}-\beta\xi^{\parallel_{\Lambda_\beta}}|_{C^{k-1,\alpha}_{-\beta,b}(\overline{Cyl}_{t_{0}})}\nonumber\\& &+|\frac{\partial^{2} \xi^{\parallel_{\Lambda_\beta}}}{\partial t^{2}}-\beta^{2}\xi^{^{\parallel_{\Lambda_\beta}}}|_{C^{k-2,\alpha}_{-\beta,b}(\overline{Cyl}_{t_{0}})}.
\\|\sigma_{2}^{-1}\xi|_{C^{k,\alpha,P^{0}}_{-\beta,\gamma,b}(\overline{Cyl}_{t_{0}})}&\triangleq &|\xi^{\perp_{\Lambda_\beta}} |_{C^{k,\alpha}_{-\beta,\gamma}(\overline{Cyl}_{t_{0}})}+|\xi^{\parallel_{\Lambda_\beta}}|_{C^{k,\alpha}_{-\beta,b}(\overline{Cyl}_{t_{0}})}. \end{eqnarray*}
\end{Def}
\begin{rmk} The $\sigma_{2}$ of the adjoint operator $L^{\star}$ in (\ref{equ relation ker and coker in L2 to weighted L2 proof of index thm}) is usually not identity, but it never affects the index or kernel.   
\end{rmk}
For all first and second-order  TID-operators, we abuse notation and denote the corresponding  operators on the link as $B_{P^{0}}$.  We need to solve the equations 
  \begin{equation}\label{equ tensor we need to solve 1st order operator} (\frac{\partial}{\partial t}-B_{P_{0}})\sigma_{2}\xi=\sigma_{1}^{-1}h,\ (\frac{\partial^{2}}{\partial t^{2}}-a_{1}\frac{\partial}{\partial t}-B_{P_{0}})\sigma_{2}\xi=\sigma_{1}^{-1}h\ \textrm{respectively}.
  \end{equation}
  Let $u=(\sigma_{2}\xi) e^{-\beta t}$, $f=(\sigma_{1}^{-1}h) e^{-\beta t}$, (\ref{equ tensor we need to solve 1st order operator})  become
  \begin{equation}\label{equ tensor we need to solve 1st order operator multiplied by ebetat} \frac{\partial u}{\partial t}-B_{P^{0}_{\beta}}u=f\ \textrm{and}\  \frac{\partial^{2} u}{\partial t^{2}} -(a_{1}-2\beta)\frac{\partial u}{\partial t}-B_{P^{0}_{\beta}} u=f\ \textrm{respectively},
   \end{equation}
where 
\begin{equation}B_{P^{0}_{\beta}}\triangleq\left\{ \begin{array}{cc}B_{P^{0}}-\beta Id\ & \textrm{when}\ P^{0}\ \textrm{is first-order} \\
B_{P^{0}}+a_{1}\beta-\beta^{2}\ & \textrm{when}\ P^{0}\ \textrm{is second-order  elliptic}.
\end{array}\right.
\end{equation}
   
For any $\Lambda\in Spec(B_{P^{0}})$ (repeated by multiplicity). Let \begin{equation}\label{equ spec P0beta}\lambda\ [\in Spec(B_{P_{\beta}^{0}})]\triangleq\left\{\begin{array}{cc}\Lambda-\beta & \textrm{when}\ P^{0}\ \textrm{is first-order},\\
\Lambda+a_{1}\beta-\beta^{2} & \textrm{when}\ P^{0}\ \textrm{is second-order  elliptic}.
\end{array}\right.\end{equation}
Let $[\phi_{\Lambda},\Lambda\in Spec(B_{P^{0}})]$ denote the  orthonormal eigen-basis of $L^{2}[Y,E]$ with respect to $B_{P^{0}}$. Abusing notation, we let $\phi_{\lambda}=\phi_{\Lambda}$. In terms of the Fourier series $u\  (f)=\Sigma_{\Lambda} u_{\lambda}\phi_{\Lambda}\  (\Sigma_{\Lambda} f_{\lambda}\phi_{\Lambda})$,  (\ref{equ tensor we need to solve 1st order operator multiplied by ebetat}) is equivalent to the ODE's 
 \begin{equation}\label{equ the simple ODE  we need to solve 1st order operator}
   \frac{d u_{\lambda}}{dt}-\lambda u_{\lambda}=f_{\lambda},\ \frac{d^{2} u_{\lambda}}{dt^{2}}-(a_{1}-2\beta)\frac{du_{\lambda}}{dt}-\lambda u_{\lambda}=f_{\lambda}\ \textrm{for all}\ \lambda\in SpecB_{P_{\beta}^{0}}\ \textrm{respectively}.  
   \end{equation}
   \begin{rmk}\label{Rmk Fourier interpretation of the norms}Let $\sigma_{2}=Id$. In terms of  the Fourier-coefficients, when $P^{0}$ is  first-order, 
\begin{eqnarray*}
|\xi|^{2}_{\widehat{W}^{1,2,P_{\beta}^{0}}_{0,\gamma,b-1}(Cyl_{t_{0}})}&=&  \Sigma_{\lambda\in Spec(B_{P^{0}_{\beta}}),\lambda\neq 0}[ (1+\lambda^{2})\int^{\infty}_{t_{0}} \xi^{2}_{\lambda}s^{2\gamma}ds+\int^{\infty}_{t_{0}} |\frac{d \xi_{\lambda}}{d s}|^{2}s^{2\gamma}ds]\\& & 
+\Sigma_{\lambda\in Spec(B_{P^{0}_{\beta}}), \lambda =0}[  \int^{\infty}_{t_{0}} \xi^{2}_{\lambda}s^{2b-2}ds+\int^{\infty}_{t_{0}} |\frac{d\xi_{\lambda}}{ds}|^{2}s^{2b}ds].
 \end{eqnarray*}
When $P^{0}$ is  second-order elliptic, using the usual $W^{2,2}-$elliptic estimate on strips, we routinely verify the following for any $\xi$ compactly supported in $Cyl_{t_{0}+\epsilon}$. 
\begin{eqnarray*}
|\xi|^{2}_{\widehat{W}^{2,2,P_{\beta}^{0}}_{0,\gamma,b-1}(Cyl_{t_{0}})}&\leq & C(\epsilon)\{\Sigma_{\lambda\in Spec(B_{P^{0}_{\beta}}),\lambda\neq 0} \int^{\infty}_{t_{0}} [(1+\lambda^{2})\xi^{2}_{\lambda}+(1+|\lambda|)|\frac{d\xi_{\lambda}}{dt}|^{2}+|\frac{d^{2}\xi_{\lambda}}{dt^{2}}|^{2}] t^{2\gamma}dt  \\& & 
+ \Sigma_{\lambda\in Spec(B_{P^{0}_{\beta}}), \lambda =0}  [\int^{\infty}_{t_{0}} \xi^{2}_{\lambda}t^{2b-2}ds +\int^{\infty}_{t_{0}}( |\frac{d\xi_{\lambda}}{dt}|^{2}+ |\frac{d^{2}\xi_{\lambda}}{dt^{2}}|^{2})t^{2b}dt]\}
 \end{eqnarray*}
\end{rmk}
\begin{rmk}  Multiplying by $e^{-\beta t}$ is a linear isomorphism:
  \begin{equation*}\label{equ ebetat multiple is an isomorphism 1}
\widehat{C}^{k,\alpha,P^{0}}_{-\beta,\gamma,b-1}
\longrightarrow  \widehat{C}^{k,\alpha,P^{0}_{\beta}}_{0,\gamma,b-1},\ \widehat{W}^{k,p,P^{0}}_{-\beta,\gamma,b-1}
\longrightarrow \widehat{W}^{k,p,P^{0}_{\beta}}_{0,\gamma,b-1},\ C^{k,\alpha}_{-\beta,\gamma,b} \longrightarrow C^{k,\alpha}_{0,\gamma,b},\ L^{2}_{-\beta,\gamma,b}\longrightarrow L^{2}_{0,\gamma,b}. \end{equation*}
 \end{rmk}

\begin{Def}\label{Def Global graph norm and the beta vector}
Suppose $\overrightarrow{\beta}=(\beta_{1},.....,\beta_{l_{0}})$, $\overrightarrow{\gamma}=(\gamma_{1},......,\gamma_{l_{0}})$,  $\overrightarrow{b}=(b_{1},......,b_{l_{0}})$  are vectors of $l_{0}-$entries. Given an $ATID$ operator $P$ over a manifold $N$ with $l_{0}$ cylindrical ends, we denote the ends by $U_{j},\ j=1....l_{0}$. We add the interior $U_{0}$ to obtain an open cover of $N$.  Using a partition of unity $\chi_{j},\ j=0....l_{0}$ subordinate to the cover, we define 
\begin{equation}\label{equ partition unity Def for norms}
|\xi|_{\widehat{W}^{k,p,P}_{-\beta,\gamma,b-1}(N)}=|\chi_{0}\xi|_{W^{k,2}(U_{0})}+\Sigma_{j=1}^{l_{0}}|\chi_{j}\xi|_{\widehat{W}^{k,p,P^{0,j}}_{-\beta_{j},\gamma_{j},b_{j}-1}(U_{j})},
\end{equation}
where $P^{0,j}$ is the limit $TID$ operator of $P$ on the $j-$th end. The same definition as (\ref{equ partition unity Def for norms}) applies to all the other norms in Definition \ref{Def graph norm} (including $\widehat{C}^{k,\alpha,P}_{-\beta,\gamma,b-1}(N)$, $C^{k,\alpha,P}_{-\beta,\gamma,b}(N)$ etc).

 When the domain is the whole  manifold, we usually hide the $N$ in the norm symbols.

\textbf{Important Convention:} When  $\beta_{1}=...=\beta_{l_{0}}=\beta$, we denote $\overrightarrow{\beta}$ (a vector) as $\beta$ (number). The same applies to  $\overrightarrow{b}$ and 
$\overrightarrow{\gamma}$. This makes the notations consistent.  \end{Def}

\begin{Def}\label{Def ATID} Let $\delta_{0}>0$ be small enough with respect to the data in Theorem \ref{Thm local invertibility} except $t_{0}$, such that the Neumann-Series in Lemma \ref{lem regularity of S0} and Theorem \ref{Thm local invertibility for ATID} converge as desired.

   Let $k_{0}\geq 10$,   $|\cdot|_{C^{k}}(t,y)\triangleq \Sigma_{0\leq i+j\leq k}|\frac{\partial}{\partial t^{i}}\nabla^{j}\cdot|(t,y)$.  We say  that $P$ satisfies the \\ $\circledS_{\beta}(l_{1},l_{2})|_{Cyl_{t_{0}}}-$condition, if the following holds for any $k\leq k_{0}$, $t\geq t_{0}+1$, $\alpha\in [0,1)$, $\xi$, and a $\delta_{0}$ small enough with respect to the data in Theorem \ref{Thm local invertibility}.
   \begin{eqnarray*}\label{equ Def Sbeta condition}
 t^{l_{1}}|(P-P^{0})\xi^{\perp_{\beta}}|_{C^{k,\alpha}(\overline{S}_{t})}\leq \delta_{0}|\xi^{\perp_{\beta}}|_{C^{k+m_{0},\alpha}(\overline{S}_{t})},\
t^{l_{1}}|(P-P^{0})\xi^{\perp_{\beta}}|_{C^{k}}(t,y)\leq \delta_{0}|\xi^{\perp_{\beta}}|_{C^{k+m_{0}}}(t,y)\nonumber\\
 t^{l_{2}}|(P-P^{0})\xi^{\parallel_{\beta}}|_{C^{k,\alpha}(\overline{S}_{t})}\leq \delta_{0}|\xi^{\parallel_{\beta}}|_{C^{k+m_{0},\alpha}(\overline{S}_{t})},\
t^{l_{2}}|(P-P^{0})\xi^{\parallel_{\beta}}|_{C^{k}}(t,y)\leq \delta_{0}|\xi^{\parallel_{\beta}}|_{C^{k+m_{0}}}(t,y)
\end{eqnarray*}

 We say that $P$ satisfies  $\circledS(l)|_{Cyl_{t_{0}}}$ if it satisfies $\circledS_{\beta}(l_{1},l_{2})|_{Cyl_{t_{0}}}$ for all $\beta$ and $l_{1}=l_{2}=l$. We say that $P$  is   $\overrightarrow{\beta}-ATID$ on $N$ if for any $i$, it satisfies $\circledS_{\beta_{i}}(0,1)|_{Cyl_{t_{0}}}$ for some $t_{0}$ on the $i-$th end. 
 \end{Def}
\begin{rmk}\label{Rmk Dependence of delta 0}By our definition, $\delta_{0}$ depends on $P^{0},\ B_{P^{0}},\ \gamma,\ \beta,\ b$ etc. 
\end{rmk}
\begin{rmk} It's easy to check $\circledS(l)|_{Cyl_{t_{0}}}$ for differential operators. In an arbitrary coordinate neighbourhood, write $P^{0}$ and $P$ as
 \begin{equation}
P^{0}=a_{0,1}(y)\frac{\partial}{\partial t}+\Sigma_{\gamma=2}^{n}a_{0,\gamma}(y)D^{\gamma};\ P=a_{1}(y,t)\frac{\partial}{\partial t}+\Sigma_{\gamma=2}^{n}a_{\gamma}(y,t)D^{\gamma}.
 \end{equation}
 Let $\delta_{1}$ be small enough with respect to the data in Theorem \ref{Thm local invertibility} (even smaller than $\delta_{0}$), then $P$ satisfies $\circledS(l)|_{Cyl_{t_{0}}}$ if the following holds $\textrm{for all}\ y,\ k\leq k_{0}+1,\ t\geq t_{0}$.
\begin{equation}\label{equ coeff decay}t^{l}|a_{\gamma}-a_{0,\gamma}|_{C^{k}}(t,y)\leq \delta_{1}\ c.f.\ \cite[(6.5)]{Lockhart}.\end{equation}
 A simple example of an $\circledS_{0}(0,1)|_{Cyl_{t_{0}}}-$operator which does not satisfy (\ref{equ coeff decay})  is $P^{0}+\delta_{0}\sin t\frac{\partial}{\partial t}$.
\end{rmk} 
 Defining  $(\square|^{Sobolev,p}_{-\overrightarrow{\beta},\overrightarrow{\gamma}, \overrightarrow{b}})$ and $(\square|^{Schauder}_{-\overrightarrow{\beta},\overrightarrow{\gamma}, \overrightarrow{b}})$ as (\ref{eqnarray the theories in Thm Fredholm}), Theorem \ref{Thm Fredholm in the cyl setting} naturally generalizes to 
\begin{thm}\label{Thm Fredholm  different weight on different ends} Let $m_{0}$, $k$, $k_{0}$, $\alpha$, $p$ be as in Theorem \ref{Thm Fredholm in the cyl setting}.  Suppose $P$  is  $\overrightarrow{\beta}-ATID$ elliptic, and $\beta_{j}$ is not $P^{0,j}-$super indicial for any $j$. Then   
\begin{itemize}
\item $(\square|^{Sobolev,p}_{-\overrightarrow{\beta},\overrightarrow{\gamma}, \overrightarrow{b}})$ is Fredholm if  for any $j$,  $b_{j}\neq 1-\frac{1}{p}$ or $\beta_{j}\ \textrm{is not}\ P^{0,j}-\textrm{indicial}$;
\item $(\square|^{Schauder}_{-\overrightarrow{\beta},\overrightarrow{\gamma}, \overrightarrow{b}})$ is Fredholm if for any $j$,  $b_{j}\neq 1$ or $\beta_{j}\ \textrm{is not}\ P^{0,j}-\textrm{indicial}$.
\end{itemize} 
\end{thm}

\textbf{Dependence of the Constants}: we follow the convention in \cite[Definition 2.16,2.17]{Myself2016bigpaper}: the "$C$" in a result (and the proof) depends on the data in the result, except the "$t_{0}$" (initial time for the cylinders). We will add subscripts when $C$ depends on $t_{0}$ or other parameters. 
 
 \begin{rmk}\label{Rmk abbreviation of norms} From now on, we hide the $P^{0}$ (or $P$) in the $\widehat{W}'s$ ($\widehat{C}'s$) . The underlying operator should be clear from the context. When $\gamma=b$,   we  abbreviate $\widehat{W}^{k,p,P^{0}}_{-\beta,\gamma,b-1}$, $\widehat{C}^{k,\alpha,P^{0}}_{-\beta,\gamma,b-1}$, $W^{k,p,P^{0}}_{-\beta,\gamma,b}$, $C^{k,\alpha,P^{0}}_{-\beta,\gamma,b}$ to  $\widehat{W}^{k,p}_{-\beta,b-1}$, $\widehat{C}^{k,\alpha}_{-\beta,b-1}$,  $W^{k,p}_{-\beta,b}$, $C^{k,\alpha}_{-\beta,b}$ respectively. When $\gamma=b=0$,   we further abbreviate $W^{k,p}_{-\beta,b}$, $C^{k,\alpha}_{-\beta,b}$ to $W^{k,p}_{-\beta}$, $C^{k,\alpha}_{-\beta}$.  
 \end{rmk}
\section{The Local inverses  \label{section Local inverse between weighted Sobolev graph-space}}\begin{Def}\label{Def Q+ and Q-}  Let $\underline{\beta}<\beta$ be the    indicial root adjacent  to $\beta$ from below (but not equal to $\beta$), and $\bar{\beta}> \beta$ be the   indicial root adjacent  to $\beta$ from above.
\end{Def}
\begin{thm}\label{Thm local invertibility} Let $P^{0}$  be a  TID-operator, and $\beta$ be not $P^{0}-$super indicial. The following holds in view of Definition \ref{Def TID operators}.

 (\textbf{i}): When $b\neq \frac{1}{2}$ or $\beta$ is not $P^{0}-$indicial,     $P^{0}:\ \widehat{W}^{m_{0},2}_{-\beta,\gamma,b-1}(Cyl_{t_{0}})\rightarrow L^{2}_{-\beta,\gamma,b}(Cyl_{t_{0}})$ admits a bounded linear  right inverse. Let $Q^{P^{0},t_{0}}_{\beta,+}$ ($Q^{P^{0},t_{0}}_{\beta,-}$) denote the right inverse when $b>\frac{1}{2}$ ($b<\frac{1}{2}$) respectively when $\beta$ is indicial, and $Q^{P^{0},t_{0}}_{\beta}$  denote the right inverse when $\beta$ is not indicial (When $\beta$ is not indicial, $Q^{P^{0},t_{0}}_{\beta,\pm}$ both mean $Q^{P^{0},t_{0}}_{\beta}$).

 (\textbf{ii}):  The following   (regularity) estimates hold. \begin{eqnarray}
& & \label{equ 1 in Thm local invertibility}
|Q^{P^{0},t_{0}}_{\beta,+}h|_{\widehat{C}^{k+m_{0},\alpha}_{-\beta,\gamma,b-1}(\overline{Cyl}_{t_{0}})}\leq C|h|_{C^{k,\alpha}_{-\beta,\gamma,b}(\overline{Cyl}_{t_{0}})}\  \textrm{when}\ b>1\ \textrm{and}\  \beta\ \textrm{is indicial}; 
\\& &\label{equ 2 in Thm local invertibility}
|Q^{P^{0},t_{0}}_{\beta,-}h|_{\widehat{C}^{k+m_{0},\alpha}_{-\beta,\gamma,0}(\overline{Cyl}_{t_{0}})}\leq C|h|_{C^{k,\alpha}_{-\beta,\gamma,b}(\overline{Cyl}_{t_{0}})} \ \textrm{when}\ b>1\ \textrm{and}\  \beta\ \textrm{is indicial};
\\& &
 \label{equ 3 in Thm local invertibility}
 |Q^{P^{0},t_{0}}_{\beta,-}h|_{\widehat{C}^{k+m_{0},\alpha}_{-\beta,\gamma,b-1}(\overline{Cyl}_{t_{0}})}\leq C|h|_{C^{k,\alpha}_{-\beta,\gamma,b}(\overline{Cyl}_{t_{0}})} \ \textrm{when}\ b<1\ \textrm{and}\ \beta\ \textrm{is indicial};
\\& &
 \label{equ 4 in Thm local invertibility}
|Q_{\beta}^{P^{0},t_{0}} h|_{C^{k+m_{0},\alpha}_{-\beta,\gamma}(\overline{Cyl}_{t_{0}})}\leq C|h|_{C^{k,\alpha}_{-\beta,\gamma}(\overline{Cyl}_{t_{0}})} \ \textrm{when}\ \beta\ \textrm{is not indicial};
\\& &
 \label{equ 5 in Thm local invertibility}
|Q^{P^{0},t_{0}}_{\beta,+}h|_{\widehat{W}^{k+m_{0},p}_{-\beta,\gamma,b-1}(\overline{Cyl}_{t_{0}})}\leq C|h|_{W^{k,p}_{-\beta,\gamma,b}(\overline{Cyl}_{t_{0}})}\  \textrm{when}\ b>1-\frac{1}{p}\ \textrm{and}\  \beta\ \textrm{is indicial};
\\& &
 \label{equ 6 in Thm local invertibility}
 |Q^{P^{0},t_{0}}_{\beta,-}h|_{\widehat{W}^{k+m_{0},p}_{-\beta,\gamma,b-1}(\overline{Cyl}_{t_{0}})}\leq C|h|_{W^{k,p}_{-\beta,\gamma,b}(\overline{Cyl}_{t_{0}})} \ \textrm{when}\ b<1-\frac{1}{p}\ \textrm{and}\ \beta\ \textrm{is indicial};
\\& &
 \label{equ 7 in Thm local invertibility}
|Q_{\beta}^{P^{0},t_{0}} h|_{W^{k+m_{0},p}_{-\beta,\gamma}(\overline{Cyl}_{t_{0}})}\leq C|h|_{W^{k,p}_{-\beta,\gamma}(\overline{Cyl}_{t_{0}})} \ \textrm{when}\ \beta\ \textrm{is not indicial}.
\end{eqnarray}
 
(\textbf{iii}): Suppose $\underline{\beta}$ is not super-indicial, then $Q_{\beta,+}^{P^{0},t_{0}}=Q^{P^{0},t_{0}}_{\underline{\beta},-}$  on $L^{2}_{-\underline{\beta},b}(Cyl_{t_{0}})$ for any $b$. \end{thm}
\textbf{Important Convention}: through-out the article, we say that $h$ is in (or not in) a space if and only if the norm of $h$ is $<\infty$ ($=\infty$), respectively. Therefore  all the estimates in Theorem \ref{Thm local invertibility} are regularity estimates. 
  \begin{thm}\label{Thm local invertibility for ATID} Suppose  $P\ \textrm{satisfies}\ \circledS_{\beta}(0,1)|_{Cyl_{t_{0}}}$ and $\beta$ is not $P^{0}-$super-indicial. Then except (\ref{equ 2 in Thm local invertibility}), $P$ also satisfies (\textbf{i}), (\textbf{ii}), (\textbf{iii}) in Theorem \ref{Thm local invertibility}  (with $P^{0}$ replaced by $P$, and $Q_{\beta,+}^{P^{0},t_{0}},\ Q_{\beta,-}^{P^{0},t_{0}},\ Q_{\beta}^{P^{0},t_{0}}$ replaced notationally by $Q_{\beta,+}^{P,t_{0}},\ Q_{\beta,-}^{P,t_{0}},\ Q_{\beta}^{P,t_{0}}$). 
  \end{thm}
  \begin{rmk}All the bounds in Theorem \ref{Thm local invertibility}, \ref{Thm local invertibility for ATID}    are independent of $t_{0}$. 
  \end{rmk}
 
  \begin{proof}[Proof of Theorem \ref{Thm local invertibility for ATID} assuming \ref{Thm local invertibility}:] We momentarily hide $t_{0},\beta,\pm$ in $Q^{P^{0},t_{0}}_{\beta,\pm}$ in each case of Theorem \ref{Thm local invertibility}. Theorem \ref{Thm local invertibility} and the $\beta-ATID$ condition (Definition \ref{Def ATID}) implies  when $\delta_{0}$ is sufficiently small,  the Neumann-Series (c.f.  \cite[Theorem 2, page 69]{Yosida})
  \begin{equation}\label{equ Neumann Series}[Id-Q^{P^{0}}(P^{0}-P)]^{-1}\triangleq \Sigma_{j=0}^{\infty}[Q^{P^{0}}(P^{0}-P)]^{j}
  \end{equation} converges to  a two-sided inverse of $Id-Q^{P^{0}}(P^{0}-P)$. Hence $Q^{P}\triangleq (\Sigma_{j=0}^{\infty}[Q^{P^{0}}(P^{0}-P)]^{j})Q^{P^{0}}$ is a right-inverse of $P$ i.e. $PQ^{P}=Id$, where we write $P=P^{0}[Id-Q^{P^{0}}(P^{0}-P)]$. \end{proof}

\begin{lem} \label{lem Hardy's inequality}(Hardy's inequality) For any $p\geq 2$,
\begin{eqnarray}\label{equ lem Hardy's inequality polynomial case} & &\int^{\infty}_{\frac{1}{10}}(t^{b-1}\int^{\infty}_{t}|f|ds)^{p}dt\leq C_{p,b}\int^{\infty}_{0}(t^{b}|f|)^{p}dt  \ \textrm{when}\ b>1-\frac{1}{p}; 
\\& &\int^{\infty}_{\frac{1}{10}}(t^{b-1}\int^{t}_{1}|f|ds)^{p}dt\leq C_{p,b}\int^{\infty}_{0}(t^{b}|f|)^{p}dt  \ \textrm{when}\ b<1-\frac{1}{p}.\label{equ lem Hardy's inequality polynomial case b<}
\end{eqnarray}
For all $b\in \R$, $p\geq 2$, $\vartheta\geq 0$, and $\mu\neq 0$, there exists a constant $C_{l_{\mu},b}$ which depends only on  $b$ and the lower bound on $|\mu|$  with the following properties. 
\begin{eqnarray}\label{equ Lem Hardy inequality exponential case} 
\mu^{p(1+\vartheta)}\int^{\infty}_{\frac{1}{10}}( e^{\mu t}t^{b}\int^{\infty}_{t}e^{-\mu s}|f|(s-t)^{\vartheta}ds)^{p}dt&\leq & C_{l_{\mu},p,b}\int^{\infty}_{0}(|f|t^{b})^{p}dt  \ \textrm{when}\ \mu>0 ; 
\\ \mu^{p(1+\vartheta)}\int^{\infty}_{\frac{1}{10}}( e^{\mu t}t^{b}\int^{t}_{1}e^{-\mu s}|f|(t-s)^{\vartheta}ds)^{p}dt &\leq & C_{l_{\mu},p,b}\int^{\infty}_{0}(|f|t^{b})^{p}dt  \ \textrm{when}\ \mu<0.\label{equ Lem Hardy inequality exponential case mu<} 
\end{eqnarray}
\end{lem}
(\ref{equ lem Hardy's inequality polynomial case})  and (\ref{equ lem Hardy's inequality polynomial case b<}) are special cases of  \cite[Theorem 330]{Hardy3}. The  proof for (\ref{equ Lem Hardy inequality exponential case}), (\ref{equ Lem Hardy inequality exponential case mu<})  is  elementary, we defer it  to the Appendix. 

 The proof of \cite[Lemma 6.37, Theorem 7.25]{GT}  (reflection about the boundary) yields
\begin{clm}\label{Clm extension of sections}Let $t_{0}\geq 2$. For any section $h\in C^{k,\alpha}(Y\times [t_{0},t_{0}+1])$ or $W^{k,p}[Y\times (t_{0},t_{0}+1)]$, there exists an extension $h_{E,t_{0}}$ such that 
\begin{itemize}\item $h_{E,t_{0}}=0$ over $(0,t_{0}-0.01)$, and $h_{E,t_{0}}=h$ when $t\geq t_{0}$;
\item $|h_{E,t_{0}}|_{C^{k,\alpha}[Y\times (0,t_{0}+1)]}\leq  C|h|_{C^{k,\alpha}[Y\times (t_{0},t_{0}+1)]}$, $|h_{E,t_{0}}|_{W^{k,p}[Y\times (0,t_{0}+1)]} \leq  C|h|_{W^{k,p}[Y\times (t_{0},t_{0}+1)]}$; 
\item $h_{E,t_{0}}$ is translation-invariant in $t_{0}$ i.e. $h_{E,t_{0}}(f)(t)=h_{E,2}(f_{t_{0}})(t-t_{0}+2)$,\\ $\textrm{where}\ f_{t_{0}}(t)=f(t+t_{0}-2).$
\end{itemize}
\end{clm}

We need to construct a linear operator $\dot{Q}^{P^{0}_{\beta}}_{\lambda}$ for each of the equations in (\ref{equ the simple ODE  we need to solve 1st order operator}), such that $u_{\lambda}\triangleq \dot{Q}^{P^{0}_{\beta}}_{\lambda}f_{\lambda}$  solves them respectively with required estimates.  Summing  the $\lambda$'s up, we obtain the desired right inverse:
\begin{equation}\label{equ the right inverse obtained by summing up the Fourier Series}
\widetilde{Q}^{P^{0}_{\beta}}_{0,\pm}f\triangleq \Sigma_{\lambda\in Spec(B_{P^{0}_{\beta}})}(\dot{Q}^{P^{0}_{\beta}}_{\lambda} f_{\lambda})\phi_{\lambda}. 
\end{equation}
When $\beta\neq 0$, it suffices to take \begin{equation}\label{equ def Q0beta = conjugation of Q00}\widehat{Q}^{P^{0}}_{\beta,\pm}\triangleq \sigma_{2}^{-1}\cdot e^{\beta t}\cdot \widetilde{Q}^{P^{0}_{\beta}}_{0,\pm}\cdot e^{-\beta t} \cdot\sigma_{1}^{-1}.\end{equation}
 \begin{proof}[\textbf{Proof of Theorem \ref{Thm local invertibility} (\textbf{i}) for first-order operators}:] We construct  the $\dot{Q}^{P^{0}_{\beta}}_{\lambda}$ as
   \begin{equation}\label{equ tabular choice of solutions 1st-order ODE}\begin{tabular}{|p{0.1cm}||p{2cm}|p{2.7cm}||p{0.1cm}||p{2cm}|p{2.7cm}|}
  \hline
&  & $u_{\lambda}(\triangleq \dot{Q}^{P^{0}_{\beta}}_{\lambda} f_{\lambda})$ & &  & $u_{\lambda}\triangleq \dot{Q}^{P^{0}_{\beta}}_{\lambda} f_{\lambda}$\\   \hline
1&   $\lambda = 0$,  $b>\frac{1}{2}$ & $-\int^{\infty}_{t}f_{\lambda}ds$ & 3& $\lambda > 0$, all b & $-e^{\lambda t}\int^{\infty}_{t}e^{-\lambda s}f_{\lambda}ds$ \\   \hline
 2&    $\lambda = 0$, $b<\frac{1}{2}$ & $\int^{t}_{1}f_{\lambda}ds$ &4&  $\lambda < 0$, all b & $e^{\lambda t}\int^{t}_{1}e^{-\lambda s}f_{\lambda}ds$  \\   \hline
  \end{tabular}
 \renewcommand\arraystretch{1.5}
  \end{equation}
  \begin{eqnarray}\label{eqnarray f compact supported}& & \textit{By Remark}\ \ref{Rmk Fourier interpretation of the norms}\ \textit{and completeness of the spaces in Definition}\ \ref{Def graph norm},\ \textit{it suffices to assume}\nonumber\\& &\nonumber  f\in C^{\infty}_{c}(Cyl_{1}),\  \textit{and only has finitely many non-zero Fourier coefficients}.\  \textit{Without} \\& &\nonumber 
  \textit{loss of generality, we only consider first-order operators, and assume that}\ \beta=0\\& &[\textit{see the derivation of}\ (\ref{equ 0 Proof of Theorem  local invertibility i})].\ \textit{The proof for second-order operators is similar}.   \end{eqnarray}
  Applying the 4 inequalities in Lemma \ref{lem Hardy's inequality} to  the 4 cases in (\ref{equ tabular choice of solutions 1st-order ODE}) respectively, we find 
  \begin{equation*}
  \int_{1}^{\infty} u^{2}_{\lambda}t^{2b-2}dt \leq C \int_{1}^{\infty} f_{\lambda}^{2}t^{2b}dt\ (\textrm{in cases 1, 2}),\
\lambda^{2}\int_{1}^{\infty} u^{2}_{\lambda}t^{2\gamma}dt \leq C \int_{1}^{\infty} f_{\lambda}^{2}t^{2\gamma}dt\ (\textrm{cases 3, 4}). 
  \end{equation*}
 Using the equation (\ref{equ the simple ODE  we need to solve 1st order operator}) to estimate $\frac{du_{\lambda}}{dt}$, we trivially obtain
 \begin{equation*}
  \int_{1}^{\infty} |\frac{d u_{\lambda}}{dt}|^{2} t^{2b}dt= \int_{1}^{\infty} f_{\lambda}^{2}t^{2b}dt\ (\textrm{in  cases 1, 2}),\  \int_{1}^{\infty} |\frac{d u_{\lambda}}{dt}|^{2} t^{2\gamma}dt=\int_{1}^{\infty} f_{\lambda}^{2}t^{2\gamma}dt\ (\textrm{cases 3, 4}).
 \end{equation*}
The above 4 estimates in the 4 cases yield
 \begin{equation}|\widetilde{Q}^{P^{0}_{\beta}}_{0,b}f^{\perp_{0}}|_{W^{m_{0},2}_{0,\gamma}(Cyl_{1})}\leq C|f^{\perp_{0}}|_{L^{2}_{0,\gamma}(Cyl_{1})};
 \ |\widetilde{Q}^{P^{0}_{\beta}}_{0,b}f|_{\widehat{W}^{m_{0},2}_{0,\gamma,b-1}(Cyl_{1})}\leq C|f|_{L^{2}_{0,\gamma,b}(Cyl_{1})}.\label{equ est for Q0 proof Thm local invertibility simpliest case}
 \end{equation} Using (\ref{equ def Q0beta = conjugation of Q00}) ($\sigma_{i}^{-1}$ are smooth) and the notation in Theorem \ref{Thm local invertibility} \textbf{i}, let $f=\sigma_{1}^{-1}h_{E,t_{0}}\in C^{\infty}_{c}(Cyl_{1})$, we obtain
\begin{equation} \label{equ 0 Proof of Theorem  local invertibility i}|\widehat{Q}^{P^{0}}_{\beta,\pm}h_{E,t_{0}}|_{\widehat{W}^{m_{0},2}_{-\beta,\gamma,b-1}(Cyl_{1})}\leq  C|h_{E,t_{0}}|_{L^{2}_{-\beta,\gamma,b}(Cyl_{1})}.\ \textrm{Let}\
Q^{P^{0},t_{0}}_{\beta,\pm}h\triangleq \widehat{Q}^{P^{0}}_{\beta,\pm}h_{E,t_{0}},
\end{equation}
 the following holds by Claim \ref{Clm extension of sections}.
\begin{eqnarray} \label{equ 1 Proof of Theorem  local invertibility i}& &|Q^{P^{0},t_{0}}_{\beta,\pm}h|_{\widehat{W}^{m_{0},2}_{-\beta,\gamma,b-1}(Cyl_{t_{0}})}\leq |\widehat{Q}^{P^{0},t_{0}}_{\beta,\pm}h_{E,t_{0}}|_{\widehat{W}^{m_{0},2}_{-\beta,\gamma,b-1}(Cyl_{2})}\nonumber\leq  C|h_{E,t_{0}}|_{L^{2}_{-\beta,\gamma,b}(Cyl_{1})}
\\&\leq &  C|h|_{L^{2}_{-\beta,\gamma,b}(Cyl_{t_{0}})}.
\end{eqnarray}
The above means $Q^{P^{0},t_{0}}_{\beta,\pm}$ is bounded. 
   \end{proof}
  Similar ideas apply to  second-order equations, we defer the detail  to the Appendix. 
By our constructions in  (\ref{equ tabular choice of solutions 1st-order ODE}), (\ref{equ tabular choice of solution 2nd order operator m neq 0}), (\ref{equ tabular choice of solution 2nd order operator a1-2beta<0}), (\ref{equ def Q0beta = conjugation of Q00}), we routinely verify Theorem \ref{Thm local invertibility} (\textbf{iii}). 

\begin{proof}[\textbf{Proof of Theorem \ref{Thm local invertibility} (\ref{equ 4 in Thm local invertibility}), (\ref{equ 7 in Thm local invertibility})}:] It suffices to apply Maz'ya-Plamenevskii's trick (\cite[Lemma 1.1, 4.1]{Mazya}).  We adopt (\ref{eqnarray f compact supported}) and assume $k=0$. Theorem \ref{Thm local invertibility} (\textbf{i}) yields
 \begin{eqnarray}
& & \int^{l+1}_{l-1}|\widetilde{Q}_{0}^{P^{0}_{\beta}}(\xi_{m}f)|^{2}_{L^{2}(Y)}t^{2\gamma}dt\leq Ce^{2\mu_{0}l}\int^{l+1}_{l-1}|\widetilde{Q}_{0}^{P^{0}_{\beta}}(\xi_{m}f)|^{2}_{L^{2}(Y)}e^{-2\mu_{0}t}t^{2\gamma}dt \nonumber
 \\&\leq & Ce^{2\mu_{0}l}\int^{\infty}_{1}|\xi_{m}f|^{2}_{L^{2}(Y)}e^{-2\mu_{0}t}t^{2\gamma}dt\nonumber
  \leq  Ce^{2\mu_{0}(l-m)}\int^{m+2}_{m-2}|f|^{2}_{L^{2}(Y)}t^{2\gamma}dt
  \\&\leq & Ce^{-2|\mu_{0}||l-m|}\sup_{m}|t^{\gamma}f|^{2}_{L^{2}(2S_{m})},\label{eqnarray 1 proof of Prop Schauder invertibility in the non-indicial case}
  \\& &|\widetilde{Q}_{0}^{P^{0}_{\beta}}f|_{L^{2}_{0,\gamma}(Cyl_{1})} \leq C|f|_{L^{2}_{0,\gamma}(Cyl_{1})},\label{eqnarray 1 proof of Prop Schauder Wkp invertibility in the non-indicial case}
 \end{eqnarray}
 where we let $|\mu_{0}|$  be small enough with respect the spectrum gap, and the sign of $-\mu_{0}$ be the same as that of $l-m$ (when $l=m$ either sign works). Summing the $m$  in (\ref{eqnarray 1 proof of Prop Schauder invertibility in the non-indicial case}) over all integers $\geq 3$, we obtain
 \begin{equation}\label{equation 0 proof of Prop Schauder invertibility in the non-indicial case}
 |\widetilde{Q}_{0}^{P^{0}_{\beta}}f|_{L^{2}_{0,\gamma}(S_{l})}\leq C \sup_{m}|t^{\gamma}f|_{L^{2}(2S_{m})}\Sigma_{m\geq 2}e^{-|\mu_{0}||l-m|}\leq C |f|_{C^{0}_{0,\gamma}(Cyl_{1})}\ \textrm{for any}\ l\geq 2. 
 \end{equation} We recall $|\widetilde{Q}_{0}^{P^{0}_{\beta}}f|_{L^{2}(S_{l})}\leq C l^{-\gamma}|\widetilde{Q}_{0}^{P^{0}_{\beta}}f|_{L^{2}_{0,\gamma}(S_{l})},\ |f|_{C^{\alpha}(S_{l})}\leq C l^{-\gamma}|f|_{C^{\alpha}_{0,\gamma}(S_{l})}$,\\   $|f|_{L^{p}(S_{l})}\leq C l^{-\gamma}|f|_{L^{p}_{0,\gamma}(S_{l})}$, and  the following standard (Schauder and $L^{p}$) estimate on $S_{l}$
 \begin{equation}\label{equation 1 proof of Prop Schauder invertibility in the non-indicial case}
 |\xi |_{C^{1,\alpha}(\frac{S_{l}}{2})}\leq C|P^{0}_{\beta}\xi|_{C^{\alpha}(S_{l})}+C|\xi|_{L^{2}(S_{l})},\
 |\xi|_{W^{1,p}(\frac{S_{l}}{2})}\leq C|P^{0}_{\beta}\xi|_{L^{p}(S_{l})}+C|\xi|_{L^{2}(S_{l})}.  
 \end{equation}
 Then we obtain from (\ref{eqnarray 1 proof of Prop Schauder invertibility in the non-indicial case}) and (\ref{eqnarray 1 proof of Prop Schauder Wkp invertibility in the non-indicial case}) that \begin{eqnarray}\label{eqnarray 2 proof of Prop Schauder invertibility in the non-indicial case}l^{\gamma}|\widetilde{Q}_{0}^{P^{0}_{\beta}}f|_{C^{1,\alpha}(\frac{S_{l}}{2})}\leq C[|f|_{C^{\alpha}_{0,\gamma}(\overline{Cyl}_{1})}+|f|_{C^{0}_{0,\gamma}(S_{l})}]\leq C|f|_{C^{\alpha}_{0,\gamma}(\overline{Cyl}_{1})},\\ 
 \label{eqnarray 2 proof of Prop Schauder Wkp invertibility in the non-indicial case}l^{\gamma}|\widetilde{Q}_{0}^{P^{0}_{\beta}}f|_{W^{1,p}(\frac{S_{l}}{2})}\leq C[|f|_{L^{p}_{0,\gamma}(S_{l})}+|\widetilde{Q}_{0}^{P^{0}_{\beta}}f|_{L^{2}_{0,\gamma}(S_{l})}]. \end{eqnarray}
 Taking $\sup_{l\geq 2}$ of (\ref{eqnarray 2 proof of Prop Schauder invertibility in the non-indicial case}) and $\Sigma_{l\geq 2}$  of (\ref{eqnarray 2 proof of Prop Schauder Wkp invertibility in the non-indicial case}), we obtain by Definition \ref{Def graph norm} and (\ref{eqnarray 1 proof of Prop Schauder Wkp invertibility in the non-indicial case}) that
 \begin{equation}\label{equ 2 proof of Schauder}
 |\widetilde{Q}_{0}^{P^{0}_{\beta}}f|_{C^{1,\alpha}_{0,\gamma}(\overline{Cyl}_{2})}\leq C|f|_{C^{\alpha}_{0,\gamma}(\overline{Cyl}_{1})},\ |\widetilde{Q}_{0}^{P^{0}_{\beta}}f|_{W^{1,p}_{0,\gamma}(Cyl_{2})}\leq C|f|_{L^{p}_{0,\gamma}(Cyl_{1})}.
 \end{equation}
 By the same argument in (\ref{equ 1 Proof of Theorem  local invertibility i}) [using (\ref{equ 2 proof of Schauder}) instead of \ref{equ 0 Proof of Theorem  local invertibility i})], we obtain (\ref{equ 4 in Thm local invertibility}) and (\ref{equ 7 in Thm local invertibility}).  \end{proof}
\begin{proof}[\textbf{Proof of Theorem \ref{Thm local invertibility} (\ref{equ 1 in Thm local invertibility}), (\ref{equ 2 in Thm local invertibility}), (\ref{equ 3 in Thm local invertibility}), (\ref{equ 5 in Thm local invertibility}), (\ref{equ 6 in Thm local invertibility})}:] We adopt (\ref{eqnarray f compact supported}). Using  Lemma \ref{lem Hardy's inequality} and (\ref{equ tabular choice of solutions 1st-order ODE}),  we find the simple estimates
  \begin{eqnarray}& &|Q_{+}^{P^{0}_{\beta}}f^{\parallel_{0}}|\leq C|\int^{\infty}_{t}f^{\parallel_{0}}ds|\leq C|f^{\parallel_{0}}|_{C^{0}_{0,b}(\overline{Cyl}_{t})}|\int^{\infty}_{t}s^{-b}ds|\leq C|f^{\parallel_{0}}|_{C^{0}_{0,b}(\overline{Cyl}_{t})}t^{1-b}\ \textrm{when}\ b>1,\nonumber
  \\& & |Q_{-}^{P^{0}_{\beta}}f^{\parallel_{0}}|\leq C|\int^{t}_{1}f^{\parallel_{0}}ds|\leq C|f^{\parallel_{0}}|_{C^{0}_{0,b}(\overline{Cyl}_{1})}|\int^{t}_{1}s^{-b}ds| \leq 
  \left\{ \begin{array}{cc} C|f^{\parallel_{0}}|_{C^{0}_{0,b}(\overline{Cyl}_{1})}t^{1-b}  & \textrm{when}\ b<1,\\
  C|f^{\parallel_{0}}|_{C^{0}_{0,b}(\overline{Cyl}_{1})} & \textrm{when}\ b>1. \end{array}\right.\nonumber
  \\& &(\int^{\infty}_{1}(t^{b-1}|Q^{P^{0}_{\beta}}_{\pm}f^{\parallel_{0}}|)^{p}dt)^{\frac{1}{p}}\leq C(\int^{\infty}_{1}(t^{\gamma}|f^{\parallel_{0}}|)^{p}dt)^{\frac{1}{p}}\ \textrm{when}\ b>(<)\ 1-\frac{1}{p}\ \textrm{respectively}.
 \end{eqnarray}
 Combining $\frac{\partial}{\partial t}Q^{P^{0}_{\beta}}_{\pm}f^{\parallel_{0}}=f^{\parallel_{0}}$, we find
 \begin{equation}\label{equ 1 in thm local invertibility in Schauder spaces}
\left\{ \begin{array}{c}  |Q^{P^{0}_{\beta}}_{+} f^{\parallel_{0}}|_{\widehat{C}^{1,\alpha}_{0,b-1}(\overline{Cyl}_{2})}\leq
 C |f^{\parallel_{0}}|_{C^{\alpha}_{0,b}(\overline{Cyl}_{1})}\ \textrm{when}\ b>1\ (b<1)\ \textrm{respectively}, \\
 |Q^{P^{0}_{\beta}}_{\pm} f^{\parallel_{0}}|_{\widehat{W}^{1,p}_{0,b-1}(Cyl_{2})}\leq
 C |f^{\parallel_{0}}|_{L^{p}_{0,b}(Cyl_{1})}\ \textrm{when}\ b> 1-\frac{1}{p}\ (b< 1-\frac{1}{p})\ \textrm{respectively}, \\
|Q^{P^{0}_{\beta}}_{-} f^{\parallel_{0}}|_{\widehat{C}^{1,\alpha}_{0,0}(\overline{Cyl}_{2})}\leq C |f^{\parallel_{0}}|_{C^{\alpha}_{0,b}(\overline{Cyl}_{1})}\ \textrm{when}\ b> 1 .\end{array}\right.
 \end{equation}
 Because $Q_{\pm}^{P^{0}_{\beta}}f^{\perp_{0}}$ is perpendicular to the kernel, (\ref{equ est for Q0 proof Thm local invertibility simpliest case}) and the proof of  (\ref{equ 4 in Thm local invertibility}), (\ref{equ 7 in Thm local invertibility}) yield
 \begin{equation}\label{equ 2 in proof thm local invertibility in Schauder spaces}
  |\widetilde{Q}_{0,\pm}^{P^{0}_{\beta}} f^{\perp_{0}}|_{C^{1,\alpha}_{0,\gamma}(\overline{Cyl}_{2})}\leq
 C |f^{\perp_{0}}|_{C^{\alpha}_{0,\gamma}(\overline{Cyl}_{1})},\
 |\widetilde{Q}_{0,\pm}^{P^{0}_{\beta}} f^{\perp_{0}}|_{W^{1,p}_{0,\gamma}(Cyl_{2})}\leq
 C |f^{\perp_{0}}|_{L^{p}_{0,\gamma}(Cyl_{1})}.  
 \end{equation}
  (\ref{equ 1 in thm local invertibility in Schauder spaces}), (\ref{equ 2 in proof thm local invertibility in Schauder spaces}) amount to (the special cases of) (\ref{equ 1 in Thm local invertibility}), (\ref{equ 2 in Thm local invertibility}), (\ref{equ 3 in Thm local invertibility}), (\ref{equ 5 in Thm local invertibility}), (\ref{equ 6 in Thm local invertibility}) with $P^{0}$ replaced by $P^{0}_{\beta}$, $\beta$ by $0$. The  argument in (\ref{equ 0 Proof of Theorem  local invertibility i}), (\ref{equ 1 Proof of Theorem  local invertibility i}) yields the desired five estimates in general. \end{proof}
 \section{Regularity and proof of Theorem \ref{Thm Fredholm in the cyl setting}, \ref{Thm Fredholm  different weight on different ends} \label{section Regularity and Fredholm}}
\begin{rmk}\label{rmk convention regularity reduction} Without loss of generality, in the proof of Claim \ref{clm boostrapping kernel of model operators}, Lemma \ref{lem regularity of S0}, and Proposition \ref{Prop regularity of Harmonic sections}, we only consider first-order operators, and assume  $k=t_{0}=1$,  $\gamma=b$ (see Remark \ref{Rmk abbreviation of norms}). The proof for the other cases is absolutely the same. Though second-order elliptic operators are more complicated [there are 2 homogeneous solutions to the second-order ODEs in (\ref{equ the simple ODE  we need to solve 1st order operator})], the desired regularity still follows in the same way.
\end{rmk}
  \begin{clm}\label{clm boostrapping kernel of model operators} Given any TID-operator $P^{0}$, suppose $\beta,\ \underline{\beta}$ are not $P^{0}-$super indicial. Then for any  $t_{0},k \geq 1$, $\epsilon>0$, the following estimate holds uniformly in $h\in kerP^{0}|_{L^{2}_{-\beta,\gamma,-\frac{1}{2}+\epsilon}(Cyl_{t_{0}})}$: $$|h|_{C^{k,\alpha}_{-\underline{\beta}}(\overline{Cyl}_{t_{0}+\epsilon})}\leq C_{t_{0}}|h|_{L^{2}_{-\beta,\gamma,-\frac{1}{2}+\epsilon}(Cyl_{t_{0}})}.$$ \end{clm}
\noindent \textit{Proof}: We adopt Remark \ref{rmk convention regularity reduction}. The condition $|h|_{L^{2}_{-\beta,-\frac{1}{2}+\epsilon}(Cyl_{1})}<\infty$ $(h\in Ker{P^{0}})$ implies  $h^{\parallel_{\beta}}=0$, then $h=\Sigma_{\lambda\in SpecB_{P^{0}},\lambda<\beta}h_{\lambda}e^{\lambda t}\phi_{\lambda}=\Sigma_{\lambda\leq \underline{\beta}}h_{\lambda}e^{\lambda t}\phi_{\lambda}$. The condition $dimEigen_{\underline{\beta}}B_{P^{0}}<\infty$ implies $|h^{\parallel_{\beta}}|_{C^{1,\alpha}_{-\underline{\beta}}(\overline{Cyl}_{1+\epsilon})}\leq C|h^{\parallel_{\beta}}|_{L^{2}_{-\beta,-\frac{1}{2}+\epsilon}(Cyl_{1})}.$
 
Since $h^{\perp_{\beta}}\in Ker \widehat{B}_{P^{0}_{\beta}}$, using  the Schauder estimate in (\ref{equation 1 proof of Prop Schauder invertibility in the non-indicial case}) like (\ref{eqnarray 2 proof of Prop Schauder Wkp invertibility in the non-indicial case}),  the rate of decay is improved i.e. 
$|h^{\perp_{\beta}}|_{C^{1,\alpha}_{-\underline{\beta}}(\overline{Cyl}_{1+\epsilon})}\leq C|h^{\perp_{\beta}}|_{L^{2}_{-\underline{\beta}}(Cyl_{1+\frac{\epsilon}{2}})} \leq C|h^{\perp_{\beta}}|_{L^{2}_{-\beta,-\frac{1}{2}+\epsilon}(Cyl_{1})}.\ \textrm{Thus}$
 \begin{eqnarray*}
 & &|h|_{C^{1,\alpha}_{-\underline{\beta}}(\overline{Cyl}_{1+\epsilon})}\leq |h^{\perp_{\beta}}|_{C^{1,\alpha}_{-\underline{\beta}}(\overline{Cyl}_{1+\epsilon})} +|h^{\parallel_{\beta}}|_{C^{1,\alpha}_{-\underline{\beta}}(\overline{Cyl}_{1+\epsilon})}
 \\&\leq & C(|h^{\perp_{\beta}}|_{L^{2}_{-\beta,-\frac{1}{2}+\epsilon}(Cyl_{1})} +|h^{\parallel_{\beta}}|_{L^{2}_{-\beta,-\frac{1}{2}+\epsilon}(Cyl_{1})})
  \\&\leq & C|h|_{L^{2}_{-\beta,-\frac{1}{2}+\epsilon}(Cyl_{1})}. 
 \end{eqnarray*} 

 \begin{Def}\label{Def S0 and S}$S^{P^{0},t_{0}}_{\beta,\pm}\triangleq Q^{P^{0},t_{0}}_{\beta,\pm}P^{0}-Id$ is bounded from $\widehat{W}^{m_{0},2}_{-\beta,\gamma,-\frac{1}{2}\pm \epsilon}(Cyl_{t_{0}})$ to itself. We verify $S^{P,t_{0}}_{\beta,\pm}\triangleq Q^{P,t_{0}}_{\beta,\pm}P-Id=(\Sigma_{j=0}^{\infty}[Q^{P^{0},t_{0}}_{\beta,\pm}(P^{0}-P)]^{j})S^{P^{0},t_{0}}_{\beta,\pm}$.\end{Def}
 \begin{lem}\label{lem regularity of S0} Let $P$ be an operator on $Cyl_{t_{0}}$ as in Definition \ref{Def ATID}.  Suppose $\beta,\ \underline{\beta}$ are not $P^{0}-$super indicial, then the following hold for any $\epsilon>0$, $t_{0}\geq 2$, $\gamma$, and $k\leq k_{0}+m_{0}-1$. 
  $$\left\{\begin{array}{cc}
 \textbf{I}:\ & |S^{P,t_{0}}_{\beta,+}\xi|_{\widehat{C}^{k,\alpha}_{-\underline{\beta},0}(\overline{Cyl}_{t_{0}})}\leq C_{t_{0}}|\xi|_{\widehat{C}^{k,\alpha}_{-\beta,\gamma,\epsilon}(\overline{Cyl}_{t_{0}})}\  \textrm{when}\ P\  \textrm{satisfies}\ \circledS (l)|_{Cyl_{t_{0}}}\ \textrm{with}\ l>1,\\
\textbf{II}:\ &   |S^{P,t_{0}}_{\beta,+}\xi|_{\widehat{W}^{m_{0},2}_{-\underline{\beta},-\frac{1}{2}-\epsilon}(Cyl_{t_{0}})}\leq C_{t_{0}}|\xi|_{\widehat{W}^{m_{0},2}_{-\beta,\gamma,-\frac{1}{2}+\epsilon}(Cyl_{t_{0}})}\  \textrm{when}\ P\  \textrm{satisfies}\ \circledS_{\beta}|_{Cyl_{t_{0}}},\\
\textbf{III}:\ &    |S^{P,t_{0}}_{\beta,+}\xi|_{\widehat{C}^{k,\alpha}_{-\underline{\beta},-\epsilon}(\overline{Cyl}_{t_{0}})}\leq C_{t_{0}}|\xi|_{\widehat{C}^{k,\alpha}_{-\beta,\gamma,\epsilon}(\overline{Cyl}_{t_{0}})}\  \textrm{when}\ P\  \textrm{satisfies}\ \circledS_{\beta}|_{Cyl_{t_{0}}},\\
\textbf{IV}:\ &    |S^{P,t_{0}}_{\beta,+}\xi|_{C^{k,\alpha}_{-\underline{\beta},-\frac{1}{2}-\epsilon}(\overline{Cyl}_{t_{0}+\epsilon})}\leq C_{t_{0}}|\xi|_{\widehat{W}^{m_{0},2}_{-\beta,\gamma,-\frac{1}{2}+\epsilon}(Cyl_{t_{0}})}\  \textrm{when}\ P\  \textrm{satisfies}\ \circledS_{\beta}|_{Cyl_{t_{0}}}.\\
  \end{array}\right.$$ 
 \end{lem}

 \begin{proof} We adopt Remark \ref{rmk convention regularity reduction} and only prove  \textbf{I} and \textbf{IV}.  \textbf{II} and  \textbf{III} are similar (to \textbf{I}).   First, we deal with the model operator $S^{P^{0},t_{0}}_{\beta,+}$. Because $S^{P^{0},2}_{\beta,+} \xi \in\ KerP^{0}$, Claim \ref{clm boostrapping kernel of model operators} yields  
 \begin{eqnarray}\label{eqnarray 1 in proof Lem regularity of S0}
& &|S^{P^{0},2}_{\beta,+}\xi|_{\widehat{C}^{1,\alpha}_{-\underline{\beta},0}(\overline{Cyl}_{2})}\leq C|S^{P^{0},2}_{\beta,+}\xi|_{\widehat{C}^{1,\alpha}_{-\beta,\epsilon}(Y\times [2,4])}+|S^{P^{0},2}_{\beta,+}\xi|_{\widehat{C}^{1,\alpha}_{-\underline{\beta},0}(\overline{Cyl}_{3})}\nonumber
\\&\leq &  C|S^{P^{0},2}_{\beta,+}\xi|_{\widehat{C}^{1,\alpha}_{-\beta,\epsilon}(\overline{Cyl}_{2})}\leq  C|Q^{P^{0},2}_{\beta,+}P^{0}\xi|_{\widehat{C}^{1,\alpha}_{-\beta,\epsilon}(\overline{Cyl}_{2})}+C|\xi|_{\widehat{C}^{1,\alpha}_{-\beta,\epsilon}(\overline{Cyl}_{2})}\nonumber
\\&\leq &   C|\xi|_{\widehat{C}^{1,\alpha}_{-\beta,\epsilon}(\overline{Cyl}_{2})}.
 \end{eqnarray}

When $P$  satisfies $\circledS (l)|_{Cyl_{1}}$ with $l>1$, $|(P^{0}-P)\eta|_{C^{\alpha}_{-\underline{\beta},l}(\overline{Cyl}_{2})}\leq C\delta_{0}|\eta|_{C^{1,\alpha}_{-\underline{\beta}}(\overline{Cyl}_{2})}\ \textrm{for any}\ \eta.$
Hence we obtain the following (by  Theorem \ref{Thm local invertibility} (\ref{equ 2 in Thm local invertibility}) and the above).
 \begin{equation}
|Q^{^{P^{0},2}}_{\beta,+}(P^{0}-P)\eta|_{\widehat{C}^{1,\alpha}_{-\underline{\beta},0}(\overline{Cyl}_{2})}=|Q^{^{P^{0},2}}_{\underline{\beta},-}(P^{0}-P)\eta|_{\widehat{C}^{1,\alpha}_{-\underline{\beta},0}(\overline{Cyl}_{2})}\leq C\delta_{0}|\eta|_{C^{1,\alpha}_{-\underline{\beta}}(\overline{Cyl}_{2})}.
 \end{equation} 
 Let $\delta_{0}$ be small enough such that $C\delta_{0}\leq \frac{1}{2}$ in the above, the Neumann series converges i.e.
 \begin{equation}\label{equ 2 in proof Lem regularity of S0}
 |(\Sigma_{j=0}^{\infty}[Q_{\beta,+}^{P^{0},2}(P^{0}-P)]^{j})\eta|_{\widehat{C}^{1,\alpha}_{-\underline{\beta},0}(\overline{Cyl}_{2})}\leq C|\eta|_{C^{1,\alpha}_{-\underline{\beta}}(\overline{Cyl}_{2})}.
 \end{equation}
Let $\eta=S^{P^{0},2}_{\beta,+}\xi$ in (\ref{equ 2 in proof Lem regularity of S0}),  Lemma \ref{lem regularity of S0} \textbf{I} follows from  (\ref{eqnarray 1 in proof Lem regularity of S0}) and Definition \ref{Def S0 and S}. 

 On \textbf{IV}, using the Schauder estimate (\ref{equation 1 proof of Prop Schauder invertibility in the non-indicial case})  and \textbf{II} (note $S^{P,2}_{\beta,+}\xi\in Ker P$), we obtain 
  \begin{equation}\label{equ 3 in proof Lem regularity of S0}
 |S^{P,2}_{\beta,+}\xi|_{C^{k,\alpha}(\frac{\epsilon S_{l}}{2})}\leq C|S^{P,2}_{\beta,+}\xi|_{L^{2}(\epsilon S_{l})}\leq Cl^{\frac{1}{2}+\epsilon} e^{\underline{\beta}l}|S^{P,2}_{\beta,+}\xi|_{L^{2}_{-\underline{\beta},-\frac{1}{2}-\epsilon}(\epsilon S_{l})}.
 \end{equation}
 By Definition \ref{Def Global graph norm and the beta vector}, the proof of \textbf{IV} is complete.
 \end{proof}
 
\begin{prop}\label{Prop regularity of Harmonic sections}Under the same conditions on $P$, $\epsilon$, $k$, $\gamma$, $\beta,\ \underline{\beta}$ in Lemma \ref{lem regularity of S0}, the following hold uniformly for any  $h\in KerP|_{\widehat{W}^{m_{0},2}_{-\beta,\gamma,-\frac{1}{2}+\epsilon}(Cyl_{t_{0}})}$. 
  $$\left\{\begin{array}{cc}
   |h|_{\widehat{C}^{k,\alpha}_{-\underline{\beta}, -\epsilon}(\overline{Cyl}_{t_{0}+\epsilon})}\leq C_{t_{0}}|h|_{\widehat{W}^{m_{0},2}_{-\beta,\gamma,-\frac{1}{2}+\epsilon}(Cyl_{t_{0}})}\ & \textrm{when}\ P\  \textrm{satisfies}\ \circledS_{\beta}|_{Cyl_{t_{0}}},\\
 |h|_{\widehat{C}^{k,\alpha}_{-\underline{\beta},0}(\overline{Cyl}_{t_{0}+\epsilon})}\leq C_{t_{0}}|h|_{\widehat{W}^{m_{0},2}_{-\beta,\gamma,-\frac{1}{2}+\epsilon}(Cyl_{t_{0}})}\ & \textrm{when}\ P\  \textrm{satisfies}\ \circledS (l)|_{Cyl_{t_{0}}}\ \textrm{with}\ l>1.\\
  \end{array}\right.$$ 
  Consequently, for any $k\leq k_{0}-2$, suppose the  signs of $\mu_{1}$ and $\mu_{2}$  are the same, we have
\begin{equation}(ker|Coker|Index)(\square|^{Sobolev,p}_{-\beta,\gamma,1-\frac{1}{p}+\mu_{1}})=(ker|Coker|Index)(\square|^{Schauder}_{-\beta,\gamma,1+\mu_{2}})\ \textrm{respectively}.
\end{equation}
Moreover, suppose $\underline{\beta}<\overline{\beta}$ are 2 adjacent indicial roots which are not $P^{0}-$super indicial, then
\begin{eqnarray*}& &(ker|Coker|Index)(\square|^{Sobolev,p}_{-\overline{\beta},\gamma_{1},1-\frac{1}{p}+\epsilon_{1}})=(ker|Coker|Index)(\square|^{Sobolev,p}_{-\beta,\gamma_{2},b})\\& &=(ker|Coker|Index)(\square|^{Sobolev,p}_{-\underline{\beta},\gamma_{3},1-\frac{1}{p}-\epsilon_{2}}).
\\& &(ker|Coker|Index)(\square|^{Schauder}_{-\overline{\beta},\gamma_{1},1+\epsilon_{1}})=(ker|Coker|Index)(\square|^{Schauder}_{-\beta,\gamma_{2},b})\\& &=(ker|Coker|Index)(\square|^{Schauder}_{-\underline{\beta},\gamma_{3},1-\epsilon_{2}}).
\end{eqnarray*}
 for any $p\geq 2$, $b$, $\gamma_{i}$ ($i=1,2,3$), $\beta\in (\underline{\beta},\overline{\beta})$, and $\epsilon_{1},\epsilon_{2}>0$.
\end{prop}
\noindent \textit{Proof}:  It's a direct corollary of Lemma \ref{lem regularity of S0}. We only prove the second assertion, the first is easier.  We adopt Remark \ref{rmk convention regularity reduction}. We note that $h\in Ker P$ implies 
\begin{equation}\label{equ 1 proof of Prop regularity of Harmonic sections}
h=-S^{P,t}_{\beta,+}h\ \textrm{for all}\ t\geq t_{0}.
\end{equation}
Let $t=t_{0}$, Lemma \ref{lem regularity of S0} \textbf{(IV)} says $|h|_{\widehat{C}^{k,\alpha}_{-\underline{\beta},-2}(\overline{Cyl}_{t_{0}+\epsilon})}\leq C|h|_{\widehat{W}^{m_{0},2}_{-\beta,-\frac{1}{2}+\epsilon}(Cyl_{t_{0}})}$. Let $t=t_{0}+\epsilon$ in (\ref{equ 1 proof of Prop regularity of Harmonic sections}), Lemma \ref{lem regularity of S0} \textbf{(I)} and the above  imply
$$|h|_{\widehat{C}^{k,\alpha}_{-\underline{\beta},0}(\overline{Cyl}_{t_{0}+\epsilon})}\leq C|h|_{\widehat{C}^{k,\alpha}_{-\underline{\beta},-2}(\overline{Cyl}_{t_{0}+\epsilon})}\leq C|h|_{\widehat{W}^{m_{0},2}_{-\beta,-\frac{1}{2}+\epsilon}(Cyl_{t_{0}})}.$$
\begin{proof}[\textbf{Proof of Theorem \ref{Thm Fredholm in the cyl setting}, \ref{Thm Fredholm  different weight on different ends}}:] With the help of Lemma  \ref{lem regularity of S0}, it is standard. We only do the argument for \ref{Thm Fredholm in the cyl setting},  and only show   the $\widehat{W}$ ($\widehat{C}$) theory for $b>1-\frac{1}{p}$ ($b>1$), respectively.

 By  \cite[(2.5) to line 2, page 421]{Lockhart} (also see \cite[Theorem 4.14]{Myself2016bigpaper}), piecing together the local right inverses in Theorem \ref{Thm local invertibility for ATID} on the ends, we obtain a global parametrix  $Q_{\beta,+,N}$ such that 
 \begin{equation}
 \left\{\begin{array}{c}Q_{\beta,+,N}P=Id+S_{left},\\
  PQ_{\beta,+,N}=Id+S_{right}\end{array}\right. ,\ S_{right}\ \textrm{is compact from}\ W^{k,p}_{-\beta,\gamma,b}\ (C^{k,\alpha}_{-\beta,\gamma,b})\ \textrm{to itself}.
 \end{equation}
Lemma \ref{lem regularity of S0} \textbf{IV} says $S_{left}$ is bounded from $\widehat{W}^{m_{0},2}_{-\beta,\gamma,b-1}$ to $\widehat{C}^{k_{0}+m_{0}-1,\alpha}_{-\underline{\beta},-1}$. Since the weight is improved, by the (cylindrical analogue of)   \cite[proof of Lemma 4.10]{Myself2016bigpaper}, $\widehat{C}^{k_{0}+m_{0}-1,\alpha}_{-\underline{\beta},-1}$ embeds compactly into  $\widehat{W}^{k+m_{0},p}_{-\beta,\gamma,b-1}$ and $\widehat{C}^{k+m_{0},\alpha}_{-\beta,\gamma,b-1}$ when $k\leq k_{0}-2$. Then $S_{left}$ is compact from $\widehat{W}^{k+m_{0},p}_{-\beta,\gamma,b-1}$ $(\widehat{C}^{k+m_{0},\alpha}_{-\beta,\gamma,b-1})$ to itself. Hence  \cite[Theorem 4.6.5]{Zhanggongqing} implies $P$ is Fredholm. \end{proof}

\section{Index \label{section index}}
 
\begin{Def} Let $B$ be an operator as in Proposition \ref{prop eta}. Let  $d_{\lambda}\triangleq dim ker(B-\lambda Id)$. Let $Eta(B)$ denote the eta-invariant defined in \cite[Theorem 3.10 (iii)]{APS}, and 
\begin{equation*}H_{\beta,B} \triangleq \left\{\begin{array}{cc} -\frac{d_{0}}{2}-\Sigma_{0>\lambda\geq \beta} d_{\lambda} & \textrm{when}\ \beta<0,\\
-\frac{d_{0}}{2} & \textrm{when}\ \beta=0,\\
\frac{d_{0}}{2}+\Sigma_{0<\lambda<\beta} d_{\lambda} & \textrm{when}\ \beta>0.\\
 \end{array}\right\}, \ H_{\beta,B,b} \triangleq \left\{\begin{array}{cc} H_{\beta,B} & \textrm{when}\ b>\frac{1}{2},\\
H_{\beta,B}+d_{\beta} & \textrm{when}\ b<\frac{1}{2}.\\
\end{array}\right. 
\end{equation*}
 
\end{Def}
\begin{thm}\label{Thm general index}  Under the  conditions in   Theorem \ref{Thm Fredholm  different weight on different ends}, suppose  in addition that  $P$  is  first-order elliptic, $\sigma_{1}$ in (\ref{equ Def P0 TID}) is an isometry, and $P$ is translation invariant on each end. Then
\begin{equation}\label{equ thm index} index P |_{\widehat{W}^{1,2}_{-\overrightarrow{\beta},\overrightarrow{b}-1}\rightarrow L^{2}_{-\overrightarrow{\beta},\overrightarrow{b}}}=\int_{X}\alpha_{0}dvol-\frac{Eta(B_{P^{0}})}{2}+\Sigma_{j=1}^{l_{0}} H_{\beta_{j},B_{P^{0,j}}^{j},b_{j}},\ \textrm{where}\end{equation}
$\alpha_{0}$ is the $0-th$ order term in the expansion of the kernel of  $e^{-t\widetilde{P}^{\star}\widetilde{P}}-e^{-t\widetilde{P} \widetilde{P}^{\star}}$, $\widetilde{P}$ is the double of $P$ on the double of $N_{0}$ ($N_{0}\sharp\bar{N}_{0}$).
\end{thm}

\begin{rmk}\label{rmk all the index can be computed} When $P$ satisfies proper conditions as  Definition \ref{Def ATID},  under the operator norm,  we can usually  deform it continuously to be translation-invariant on each end. Thus the index of $P$ can still be computed by deformation invariance.  
\end{rmk}
\begin{proof}[Proof of Theorem \ref{Thm general index}:]Without loss of generality, we assume $\sigma_{2}=Id$.  We recall the "Extended $L^{2}-$sections" defined in  \cite[first paragraph of page 58]{APS}, and note that the dual of $L^{2}_{0,b}$ is isomorphic to $L^{2}_{0,-b}$.   Using Claim \ref{clm boostrapping kernel of model operators} for both $L$ and $L^{\star}=-(\frac{\partial}{\partial t}+B)\sigma^{-1}_{1}$, we have  
\begin{equation}\label{equ relation ker and coker in L2 to weighted L2 proof of index thm}
\begin{array}{c}Ker L|_{\widehat{W}^{1,2}_{0,b-1}}=Ker L|_{L^{2}(N,E)},\ Ker L^{\star}|_{L^{2}_{0,-b}}=Ker L^{\star}|_{\textrm{Extended}\ L^{2}(N,F)}\ \textrm{when}\ b>\frac{1}{2};\\
Ker L|_{\widehat{W}^{1,2}_{0,b-1}}=Ker L|_{\textrm{Extended}\ L^{2}(N,E)},\ Ker L^{\star}|_{L^{2}_{0,-b}}=Ker L^{\star}|_{ L^{2}(N,F)}\ \textrm{when}\ b<\frac{1}{2}.  
\end{array}
\end{equation}
\begin{equation}\label{equ 1 proof of index thm}
Index L|_{\widehat{W}^{1,2}_{0,b-1}\rightarrow L^{2}_{0,b}}=\left\{ \begin{array}{c}h(E)-h(F)-h_{\infty}(F) \ \textrm{when}\ b>\frac{1}{2};\\
 h(E)-h(F)+h_{\infty}(E) \ \textrm{when}\ b<\frac{1}{2},  
\end{array}\right.
\end{equation}
where $h(E),\ h(F),\ h_{\infty}(F),\ h_{\infty}(F)$ are defined in \cite[Corollary (3.14)]{APS}. 
Assuming $\overrightarrow{\beta}=0$ for all $j$, (\ref{equ thm index}) follows from  \cite[Corollary (3.14), (3.25)]{APS}. 

By Proposition \ref{Prop regularity of Harmonic sections} (c.f. Remark \ref{Rmk ker coker index do not change}),
 \begin{equation}\label{equ 2 proof of index thm}
Index (P|^{Sobolev,2}_{-\beta,b})=\left\{ \begin{array}{c}Index (P|^{Sobolev,2}_{-\beta+\epsilon}) \ \textrm{when}\ b>\frac{1}{2};\\
 Index (P|^{Sobolev,2}_{-\beta-\epsilon}) \ \textrm{when}\ b<\frac{1}{2}. 
\end{array}\right. \epsilon>0\ \textrm{and  small}.
\end{equation}

  The   index change formula of Lockhart-McOwen \cite[Theorem 8.1]{Lockhart} means for any $i_{0}$,  let $\beta_{i},\ i\neq i_{0}$ be unaltered, and $\beta_{i_{0}}$ go across an eigen-value $\lambda$ of $B_{i_{0}}$ (from $\lambda+\epsilon$ to $\lambda-\epsilon$), the index decreases by $dimker(B^{i_{0}}_{P^{0}}-\lambda Id)$. The proof  for general $\overrightarrow{\beta}$  is   complete with the help of   (\ref{equ 2 proof of index thm}) [and (\ref{equ thm index}) for $\overrightarrow{\beta}=0$].\end{proof}

\begin{prop}\label{prop eta}In the setting of Definition \ref{Def TID operators},  suppose $B$ is a self-adjoint first-order elliptic differential operator on $E\rightarrow Y$. Then $Eta(B-\beta Id)=Eta(B)-2H_{\beta,B}-d_{\beta}$.
\end{prop}
\begin{proof} We form the full cylinder $Y\times (-\infty,+\infty)_{t}$  and consider $P=\frac{\partial}{\partial t}-B$. We consider the  open cover $End_{+}$, $Y\times (-4,4)$, $End_{-}$, where $End_{+}=Cyl_{3}$ and $End_{-}=Y\times (-\infty,-3)$. In $End_{-}$, under the coordinate $s=-t$, $P=-\frac{\partial}{\partial s}-B=-[\frac{\partial}{\partial s}-(-B)]$.  Let  $-\overrightarrow{\beta}=(-\beta,0)$ ($e^{-\beta t}$ on $End_{+}$ and $1$ on $End_{-}$), Theorem \ref{Thm general index} says when $b>\frac{1}{2}$ that
\begin{equation}\label{equ 1 proof prop eta invariant}index P |_{\widehat{W}^{1,2}_{-\overrightarrow{\beta},b-1}\rightarrow L^{2}_{-\overrightarrow{\beta},b}}=\int_{X}\alpha_{0}dvol-\frac{Eta(B)}{2}-\frac{Eta(-B)}{2}-\frac{d_{0}}{2}+ H_{\beta,B}.\end{equation}

On the other hand, let $\rho$ be a smooth function which is equal to  $e^{\beta t}$ on $End_{+}$, and $1$ on $End_{-}$, the conjugation $P_{\rho}=\frac{1}{\rho}\cdot P\cdot\rho: \widehat{W}^{1,2}_{0,b-1}\rightarrow L^{2}_{0,b}$ has the same index as $P$. Noting 
$$P_{\rho}=\left\{\begin{array}{cc} \frac{\partial}{\partial t}-(B-\beta Id) & \textrm{on}\ End_{+}\\ 
 P& \textrm{on}\ End_{-},
\end{array}\right.$$
then Theorem \ref{Thm general index} says the following for $P_{\rho}$. 
\begin{equation}\label{equ 2 proof prop eta invariant}index P_{\rho} |_{\widehat{W}^{1,2}_{0,b-1}\rightarrow L^{2}_{0,b}}=\int_{X}\alpha_{0}dvol-\frac{Eta(B-\beta Id)}{2}-\frac{Eta(-B)}{2}-\frac{d_{\beta}}{2}-\frac{d_{0}}{2}.\end{equation}
The equality between (\ref{equ 1 proof prop eta invariant}) and (\ref{equ 2 proof prop eta invariant})
yields the desired identity.
\end{proof}
\section{Applications \label{section Applications}}\begin{Def}\label{Def convergence} Let $\Gamma$ be a section or connection of $E\rightarrow Cyl_{0}$. Suppose $\Gamma$ satisfies an elliptic equation of order $m_{0}$. For any $C_{0}>0$, $\tau\geq 0$, and real number $b$, we say that $\Gamma$ converges to $\Gamma_{0}$ (at least) at the rate $\frac{C_{0}e^{-\tau t}}{t^{b}}$ [or $O(\frac{e^{-\tau t}}{t^{b}})$], if $\Gamma_{0}$ is a section or connection  of $E\rightarrow Y$ respectively and $\limsup_{t\rightarrow \infty}\frac{e^{-\tau t}}{t^{b}}|\Gamma-\Gamma_{0}|_{C^{m_{0},\alpha}(S_{t})}< C_{0}\ (\textrm{or}\ <\infty)\ \textrm{respectively}$.
 
 When $\tau>0$, We say that the convergence is exponential. When $\tau=0$, We say that it's polynomial. We only prove \ref{Cor regularity Yang Mills} in full detail, \ref{Cor Minimal surfaces} follows similarly. 
 \end{Def}
\subsection{Yang-Mills connections: proof for Corollary \ref{Cor regularity Yang Mills} \label{section Yang Mills application}}
Denoting $A-A_{O}$ by $a$, the YM-equation is $0=d_{A}^{\star}F_{A}=d_{A_{O}+a}^{\star}(d_{A_{O}}a+F_{A_{O}}+\frac{1}{2}[a,a])$. 
\begin{equation} \label{equ Yang Mills equation}
\textrm{Thus}\ \Delta_{A_{O},Hodge}a+(-1)^{n+1}\star[a,\star F_{A_{O}}]=Q_{YM}(a)-d_{A_{O}}^{\star}F_{A_{O}}\ \textrm{assuming}\ d^{\star}_{A_{O}}a=0,
\end{equation}
where $\Delta_{A_{O},Hodge}=d^{\star}_{A_{O}}d_{A_{O}}+d_{A_{O}}d^{\star}_{A_{O}}$  is the Hodge Laplacian, and 
\begin{equation}Q_{YM}(a)=-\frac{d_{A_{O}}^{\star}[a,a]}{2}+(-1)^{n}\star[a,\star d_{A_{O}}a ]+\frac{(-1)^{n}}{2}\star[a,\star [a,a] ].\end{equation} 
We note $Q_{YM}$  is quadratic in $a$.  We routinely verify  
\begin{equation}\nonumber(-1)^{n+1}\star[a,\star F_{A_{O}}]=F_{A_{O}}\underline{\otimes}a\ (\cite[Definition    \ 3.2]{Myself2016bigpaper}),\ \Delta_{A_{O},Hodge,g_{E}}=\nabla^{\star_{E}}_{A_{O}}\nabla_{A_{O}}+ F_{A_{O}}\underline{\otimes}_{g_{E}}a. \end{equation} 
Then  in cylindrical coordinates ($r=e^{-t}$),  let $a=vdt+\theta$ ($\theta$ does not contain $dt$), we find \begin{eqnarray}& &\Delta_{A_{O},Hodge,E}a+(-1)^{n+1}\star_{E}[a,\star_{E} F_{A_{O}}] \nonumber= \nabla^{\star_{E}}_{A_{O}}\nabla_{A_{O}}+2 F_{A_{O}}\underline{\otimes}_{g_{E}}a
\\&=&-e^{2t}\left|\begin{array}{cc} dt &0  \\ 0 & Id 
 \end{array}\right| \{\frac{\partial^{2}}{\partial t^{2}}-(n-4)\frac{\partial}{\partial t}-B\}\left[\begin{array}{c}v \\ \theta  \end{array}\right].\ \ \textrm{Let}\ Y=\mathbb{S}^{n-1},\ B\ \textrm{is}\label{equ 1 proof Yang Mills Cor}
 \\&B\left[\begin{array}{c}v \\ \theta  \end{array}\right]&\triangleq \left[\begin{array}{c}\nabla^{\star_{Y}}\nabla v-2d^{\star}b-2(n-2)v\\ \nabla^{\star_{Y}}\nabla \theta-2dv-2F_{A_{O}}\underline{\otimes}_{Y}\theta-(n-2) \theta  \end{array}\right]\ (\textrm{c.f.}\ \cite[(22), (24)]{Myself2016bigpaper}). \label{equ 2 proof Yang Mills Cor}
\end{eqnarray}
  As an usual strategy for non-linear equations, we view $Q_{YM}$ as a linear operator defining 
\begin{equation}\label{equ 3 proof Yang Mills Cor}\widehat{Q}_{YM,a}(b)=-\frac{d_{A_{O}}^{\star}[b,a]}{2}+(-1)^{n}\star[b,\star d_{A_{O}}a ]+\frac{(-1)^{n}}{2}\star[b,\star [a,a] ].\end{equation}
Hence $Q_{YM}(a)=\widehat{Q}_{YM,a}(a)$, and we can write 
(\ref{equ Yang Mills equation}) in cylindrical coordinates as 
\begin{equation}\label{equ 4 proof Yang Mills Cor}
P_{YM}(a)\triangleq \{e^{-2t}(\Delta_{A_{O},Hodge}+(-1)^{n+1}\star[\cdot,\star F_{A_{O}}])-e^{-2t}\widehat{Q}_{YM,a}\}a= -e^{-2t}d_{A_{O}}^{\star}F_{A_{O}}.
\end{equation}
The conditions on $g$ , $a$, (\ref{equ 1 proof Yang Mills Cor}), and (\ref{equ 2 proof Yang Mills Cor})  implies  $P_{YM}$ is ATID in the cylindrical coordinates. Moreover, 
that  $d_{A_{O}}^{\star_{E}}F_{A_{O}}=0$ (tangent connection is Yang-Mills) implies  $-e^{-2t}d_{A_{O}}^{\star}F_{A_{O}}\in C^{\alpha}_{1,0}(\overline{Cyl}_{-\log R})$ (exponential decay). Applying $Q^{P_{YM},-2\log R}_{0,+}$ (=$Q^{P_{YM},-2\log R}_{\underline{0},-}$) to both sides of (\ref{equ 4 proof Yang Mills Cor}), Lemma \ref{lem regularity of S0} $\textbf{III}$ implies that  $a$ decays exponentially. This in turn means $P_{YM}$ satisfies $\circledS(l)$ for all $l>1$. Then  Corollary \ref{Cor regularity Yang Mills} $I$  follows from applying   Proposition \ref{Prop regularity of Harmonic sections} to $a+Q^{P_{YM},-2\log R}_{0,+}e^{-2t}d_{A_{O}}^{\star}F_{A_{O}}\in KerP_{YM}$ and  Theorem \ref{Thm local invertibility for ATID} to $Q^{P_{YM},-2\log R}_{0,+}e^{-2t}d_{A_{O}}^{\star}F_{A_{O}}$.

  Combining $I$ and Lemma \ref{lem Coulomb gauge}, the proof of $II$ is complete.  
\subsection{Minimal Surfaces: proof for Corollary \ref{Cor Minimal surfaces}}
 By \cite[Section 5, page 247]{AdamSimon}, in the cylindrical coordinate $t=\log \frac{1}{|x|}$, the graph type minimal sub-manifold equation can be written as (\ref{equ 1 in proof Cor Minimal surface}) in  terms of a section $u$ to $T^{\perp}\Sigma|_{\mathbb{S}^{N}}$ (the normal bundle of $\Sigma$ in $\mathbb{S}^{N}$). We note (the transition functions of) $T^{\perp}\Sigma|_{\mathbb{S}^{N}}$ does not depend on $t$ (or $|x|$), then this bundle is in the case considered by Definition \ref{Def TID operators}. 
\begin{equation}\label{equ 1 in proof Cor Minimal surface}
P_{MSM}u\triangleq Lu+\mathfrak{R}(u)+(M_{\Sigma}-L_{\Sigma})u=0,\ \textrm{where}\ 
\end{equation}
\begin{itemize}
\item $Lu=u^{\prime\prime}-(n+1)u^{\prime}+L_{\Sigma}u$ (see \cite[Page 565]{LSimon} and  \cite[(1.8)]{AdamSimon}),
\item $\mathfrak{R}(u)$ satisfies  \cite[(1.9)]{LSimon},  $L_{\Sigma}$ is the linearisation of $M_{\Sigma}$ (see \cite[(5.2) and Page 248]{AdamSimon}.
\end{itemize}
 By the idea in (\ref{equ 3 proof Yang Mills Cor}), we  can view $\mathfrak{R}+M_{\Sigma}-L_{\Sigma}$ as a linear operator i.e.
\begin{eqnarray*}\label{equ 1 proof of Cor Minimal Surface}
(\mathfrak{R}+M_{\Sigma}-L_{\Sigma})(v)\triangleq Q_{MSM}(v)&=&a(x,t,u, \nabla u,\nabla^{2}u,u^{\prime})\cdot \nabla^{2}v+b(x,t,u, \nabla u,\nabla^{2}u)\cdot v^{\prime}\\& &+c(x,t,u, \nabla u,\nabla^{2}u)\cdot \nabla v^{\prime}+d(x,t,u, \nabla u,\nabla^{2}u)\cdot  v^{\prime\prime}\nonumber
\end{eqnarray*}
where $a,\ b,\ c,\ \nabla$ can be found in \cite[the paragraph enclosing (1.9); between line 1 in page 565 and (7.35)]{LSimon}.   Exactly as  the proof of Corollary \ref{Cor regularity Yang Mills},  $u$ decays as $\frac{\delta_{0}}{t}$ implies $P_{MSM}$ is $0-$ATID. Thus Proposition \ref{Prop regularity of Harmonic sections} [and the discussion below (\ref{equ 4 proof Yang Mills Cor})]   yields the desired improvement.
 \section{Appendix}\begin{proof}[\textbf{Proof of Theorem \ref{Thm local invertibility} (\textbf{i}) for second-order operators}:] We mainly focus on the case when $a_{1}-2\beta>0$.  We solve the second-order ODE in (\ref{equ the simple ODE  we need to solve 1st order operator}) according to the following.
\begin{equation}\label{equ tabular choice of solution 2nd order operator m neq 0}  \begin{tabular}{|p{0.1cm}|p{2.5cm}|p{9cm}|}
  \hline
 &   $m\triangleq a_{1}-2\beta$. Assume $m>0$.    & Corresponding solution $u_{\lambda}\triangleq \dot{Q}^{P^{0}_{\beta}}_{\lambda}f_{\lambda}$ and the derivative. $Det\triangleq m^{2}+4\lambda$, \ \    $\mu\triangleq \frac{\sqrt{|Det|}}{2}$, \ $\mu^{+}\triangleq \frac{m}{2}+\mu$, $\mu^{-}\triangleq \frac{m}{2}-\mu$  \\   \hline

 1& $Det >0$,   $\lambda\neq 0$ & $u_{\lambda}=-\frac{1}{\sqrt{Det}}[e^{\mu^{+}t}\int^{\infty}_{t}e^{-\mu^{+}s}f_{\lambda}ds+e^{\mu^{-}t}\int^{t}_{1}e^{-\mu^{-}s} f_{\lambda}ds]$, $u^{\prime}_{\lambda}=-\frac{1}{\sqrt{Det}}[\mu^{+}e^{\mu^{+}t}\int^{\infty}_{t}e^{-\mu^{+}s}f_{\lambda}ds+\mu^{-}e^{\mu^{-}t}\int^{t}_{1}e^{-\mu^{-}s} f_{\lambda}ds]$  \\   \hline
  2&    $\lambda=0$,  $b>\frac{1}{2}$ & $u_{\lambda}=\frac{1}{\sqrt{Det}}\{-e^{\mu^{+}t}\int^{\infty}_{t}e^{-\mu^{+}s}f_{\lambda}ds+\int^{\infty}_{t} f_{\lambda}ds\}$, $u^{\prime}_{\lambda}=-\frac{\mu^{+}}{\sqrt{Det}}e^{\mu^{+}t}\int^{\infty}_{t}e^{-\mu^{+}s}f_{\lambda}ds$  \\   \hline

 3&    $\lambda=0$,  $b<\frac{1}{2}$ & $u_{\lambda}=\frac{1}{\sqrt{Det}}\{-e^{\mu^{+}t}\int^{\infty}_{t}e^{-\mu^{+}s}f_{\lambda}ds-\int^{t}_{1} f_{\lambda}ds\}$,\ $u^{\prime}_{\lambda}=-\frac{\mu^{+}}{\sqrt{Det}}e^{\mu^{+}t}\int^{\infty}_{t}e^{-\mu^{+}s}f_{\lambda}ds$  \\   \hline
4&$Det =0$ &    $u_{\lambda}= e^{\frac{mt}{2}}\int^{\infty}_{t}e^{-\frac{ms}{2}}f_{\lambda}(s-t)ds$  \\    \hline

5&$Det<0$  &   $u_{\lambda}=\frac{e^{\frac{mt}{2}}}{\mu} \{\int^{\infty}_{t}e^{-\frac{ms}{2}}f_{\lambda}[\cos\mu t\sin(\mu s)- \sin\mu t \cos(\mu s)]ds\}$ \\    \hline
  \end{tabular}
 \renewcommand\arraystretch{1.5}
  \end{equation}

\textbf{Case 1, 5 in (\ref{equ tabular choice of solution 2nd order operator m neq 0}):}  When $\lambda\neq 0$, we have $\frac{\sqrt{|\lambda|}}{C}\leq \mu^{+}\leq C\sqrt{|\lambda|}$,  $-\frac{\sqrt{|\lambda|}}{C}\leq \mu^{-}\leq -C\sqrt{|\lambda|}$. We use (\ref{equ Lem Hardy inequality exponential case}), (\ref{equ Lem Hardy inequality exponential case mu<}) to estimate $u_{\lambda}$ and $u^{\prime}_{\lambda}$ termwise, then use the second-order equation in (\ref{equ the simple ODE  we need to solve 1st order operator}) to estimate $u^{\prime\prime}_{\lambda}$. Then  we obtain the following for Case 1 and 4. 
\begin{eqnarray}\label{eqnarray 4 in proof of main local Sobolev thm for 2nd order operators}
 \int_{1}^{\infty} |u^{\prime\prime}_{\lambda}|^{2}t^{2\gamma}dt+|\lambda|\int_{1}^{\infty} |u^{\prime}_{\lambda}|^{2}t^{2\gamma}dt+\lambda^2\int_{1}^{\infty} u^{2}_{\lambda}t^{2\gamma}dt\leq   C \int_{1}^{\infty} f_{\lambda}^{2}t^{2\gamma}dt. 
  \end{eqnarray}
Since $|\cos x|,\ |\sin x|\leq 1$ for all $x$, and $\frac{1}{\sqrt{\lambda}}\leq C$ when $\lambda\neq 0$ [$Spec(B)$ is discrete], (\ref{eqnarray 4 in proof of main local Sobolev thm for 2nd order operators}) holds for Case 5 in similar way. 

\textbf{Cases 2,3 in (\ref{equ tabular choice of solution 2nd order operator m neq 0}):} We still 
do the termwise estimates. Using  (\ref{equ Lem Hardy inequality exponential case}) for $e^{\mu^{+}t}\int^{\infty}_{t}e^{-\mu^{+}s}f_{\lambda}ds$,   (\ref{equ Lem Hardy inequality exponential case mu<}) for $e^{-\mu^{+}t}\int^{t}_{1}e^{\mu^{+}s}f_{\lambda}ds$, and  (\ref{equ lem Hardy's inequality polynomial case}) for $\int^{\infty}_{t}f_{\lambda}ds$, we obtain
 \begin{equation}\label{equ 5 in proof of main local Sobolev thm for 2nd order operators}
  \int_{1}^{\infty} |u^{\prime\prime}_{\lambda}|^{2}t^{2b}dt+\int_{1}^{\infty} |u^{\prime}_{\lambda}|^{2}t^{2b}dt+ \int_{1}^{\infty} u^{2}_{\lambda}t^{2b-2}dt \leq C \int_{1}^{\infty} f_{\lambda}^{2}t^{2b}dt.
  \end{equation}
 Remark \ref{Rmk Fourier interpretation of the norms},   (\ref{eqnarray 4 in proof of main local Sobolev thm for 2nd order operators}), (\ref{equ 5 in proof of main local Sobolev thm for 2nd order operators})   amount to $|\widetilde{Q}^{P^{0}_{\beta}}_{0,b}f|_{\widehat{W}^{2,2}_{0,\gamma,b-1}(Cyl_{1})}\leq  C|f|_{L^{2}_{0,\gamma,b}(Cyl_{1})}$. The argument between (\ref{equ est for Q0 proof Thm local invertibility simpliest case}) and (\ref{equ 1 Proof of Theorem  local invertibility i}) gives the desired local right inverse $Q^{P^{0},t_{0}}_{\beta,b}$  with the desired bound. 
 
 The proof when $a_{1}-2\beta<0$ is by the same analysis, but to the following table of solutions  [using the definitions in the first row of (\ref{equ tabular choice of solution 2nd order operator m neq 0}) but assuming $m=a_{1}-2\beta<0$].
 \begin{equation}\label{equ tabular choice of solution 2nd order operator a1-2beta<0}  \begin{tabular}{|p{0.1cm}|p{2.4cm}|p{9cm}|}
  \hline
  1& $Det >0$,   $\lambda\neq 0$ & the same as  Case 1 in (\ref{equ tabular choice of solution 2nd order operator m neq 0}) \\   \hline
  2&    $\lambda=0$,  $b>\frac{1}{2}$ & $u_{\lambda}=\frac{1}{\sqrt{Det}}\{-\int^{\infty}_{t}f_{\lambda}ds-e^{\mu_{-}t}\int^{t}_{1} e^{-\mu_{-}t}f_{\lambda}ds\}$ \\   \hline

 3&    $\lambda=0$,  $b<\frac{1}{2}$ & $u_{\lambda}=\frac{1}{\sqrt{Det}}\{-e^{\mu^{-}t}\int^{t}_{1}e^{-\mu^{-}s}f_{\lambda}ds+\int^{t}_{1} f_{\lambda}ds\}$  \\   \hline
4&$Det =0$ &    $u_{\lambda}= e^{\frac{mt}{2}}\int^{t}_{1}e^{-\frac{ms}{2}}f_{\lambda}(t-s)ds$  \\    \hline

5&$Det<0$ &   $u_{\lambda}= -\frac{e^{\frac{mt}{2}}}{\mu} \{\int^{t}_{1}e^{-\frac{ms}{2}}f_{\lambda}[\cos\mu t\sin(\mu s)- \sin\mu t \cos(\mu s)]ds\}$ \\    \hline
  \end{tabular}
 \renewcommand\arraystretch{1.5}
  \end{equation}

The proof when $m=a_{1}-2\beta=0$ and $\lambda> 0$ is much easier: we only need to  use (exactly)  the solutions in  Case 1 of (\ref{equ tabular choice of solution 2nd order operator a1-2beta<0}). Case 2--5 are super-indicial (which  we don't consider).  \end{proof}

 In the cone setting as Corollary \ref{Cor regularity Yang Mills}, let $\nabla_{A_{O}}^{\star}\nabla_{A_{O}}$ denote the rough Laplacian on $\Omega^{0}(adE)$.  By \cite[(17)]{Myself2016bigpaper}, in cylindrical coordinate $t=-\log r$,
 \begin{equation}\label{equ cyl formula for rough laplacian on 0forms}
 \nabla_{A_{O}}^{\star}\nabla_{A_{O}}\chi=e^{2t}\{\frac{d^{2}\chi}{dt^{2}}-(n-2)\frac{d\chi}{dt}+\Delta_{A_{O},\mathbb{S}^{n-1}}\chi\}, 
 \end{equation}
 where $\Delta_{A_{O},\mathbb{S}^{n-1}}$ is the rough Laplacian  on the link.  Let  $C^{k,\alpha}_{\gamma,b,cone}$ denote the weighted Schauder-spaces in the cone setting defined in \cite[Definition 2.10]{Myself2016bigpaper}. On $0-$forms, we note that pulling back  is an isomorphism $C^{k,\alpha}_{\gamma,b,cone}[B_{O}(R)]\rightarrow C^{k,\alpha}_{-\gamma,b}(Cyl_{-\log R})$. 
 \begin{Def} \label{Def Irreducible connections} We say that $A_{O}$ is irreducible if every parallel section to $\Omega^{0}(ad E)$ is $0$ everywhere. This means $0$ is not an eigenvalue of $\Delta_{A_{O},\mathbb{S}^{n-1}}$. 
 \end{Def}
 \begin{lem}\label{lem Coulomb gauge}(Existence of Coulomb gauge) Under the setting in the first paragraph of Corollary \ref{Cor regularity Yang Mills} (for any $n\geq 3$), suppose $A_{O}$ is irreducible. Suppose $A$ is a  $C^{2}-$connection on $B_{O}(R)\setminus O$ such that
 \begin{equation}
 |A-A_{O}|_{C^{1,\alpha}_{1,b,cone}[B_{O}(R)])}<\delta_{0}\ \textrm{for some}\ b\geq 0.
 \end{equation}
 Then there exists a gauge $s$ on $B_{O}(R)$ such that 
 \begin{itemize}
 \item $|s-Id|_{C^{2,\alpha}_{0,b,cone}[B_{O}(R)]}\leq C|A-A_{O}|_{C^{1,\alpha}_{1,b,cone}[B_{O}(R)])}\leq C\delta_{0}$,
 \item $d^{\star}_{A_{O}}[s(A)-A_{O}]=0$ in a smaller (truncated) ball $B_{O}(R^{\prime})\setminus O$, and \begin{equation}
 |S(A)-A_{O}|_{C^{1,\alpha}_{1,b,cone}[B_{O}(R^{\prime})-O])}\leq C|A-A_{O}|_{C^{1,\alpha}_{1,b,cone}[B_{O}(R)-O])}\leq C\delta_{0}.
 \end{equation} 
 \end{itemize}
\end{lem}
\begin{proof} Let $a\triangleq A-A_{O}$. By Theorem \ref{Thm local invertibility for ATID} and paragraph enclosing (\ref{equ cyl formula for rough laplacian on 0forms}), for any $b\geq 0$ and $R^{\prime}$  small enough, $\nabla_{A}^{\star}\nabla_{A}$ is invertible: $C^{2,\alpha}_{0,b,cone}[B(R^{\prime})]|_{\Omega^{0}(adE)}\longrightarrow\ C^{\alpha}_{2,b,cone}[B(R^{\prime})]|_{\Omega^{0}(adE)}$. Writing $s=e^{-\chi}$, then
\begin{equation}
d^{\star}_{A_{O}}[s(A)-A_{O}]=d^{\star}_{A_{O}}[s^{-1}as+s^{-1}d_{A_{O}}s]=d^{\star}_{A_{O}}[e^{-\chi}a e^{\chi}+e^{-\chi}d_{A_{O}}e^{\chi}]
\end{equation}
is a continuously differentiable map:
$$C^{1,\alpha}_{1,b,cone}[B_{O}(R)]|_{\Omega^{1}(adE)}\times C^{2,\alpha}_{0,b,cone}[B_{O}(R)]|_{\Omega^{0}(adE)}\longrightarrow C^{\alpha}_{2,b,cone}[B_{O}(R)]|_{\Omega^{0}(adE)}.$$ Then the proof is complete  by the standard argument in \cite[Proposition 2.3.4]{DonaldsonKronheimer}.
\end{proof}

  \begin{proof}[\textbf{Proof of  (\ref{equ Lem Hardy inequality exponential case}),\ (\ref{equ Lem Hardy inequality exponential case mu<})}:] For any $s\geq \frac{1}{10}$, leaving  the proof to the readers, we have
 \begin{equation}\label{equ 1 in clm in proof lem hardy}|\int^{s}_{1}e^{\mu t}t^{d}dt|\leq C_{l_{\mu},d}\frac{e^{\mu s}s^{d}}{\mu}\ \textrm{when}\ \mu>0,\ \int^{\infty}_{s}e^{\mu t}t^{d}dt\leq C_{l_{\mu},d}\frac{e^{\mu s}s^{d}}{-\mu}\ \textrm{when}\ \mu<0. 
 \end{equation}
For (\ref{equ Lem Hardy inequality exponential case}),  H\"older's inequality and the change of variable $z=s-t$ yield\begin{eqnarray}& &\nonumber(\int^{\infty}_{t}e^{-\mu s}f(s-t)^{\vartheta}ds)^{p}\leq (\int^{\infty}_{t}e^{-\mu s}f^{p}ds)(\int^{\infty}_{t}e^{-\mu s}(s-t)^{\frac{p\vartheta}{p-1}}ds)^{p-1}\\
&\leq &C(\int^{\infty}_{t}e^{-\mu s}f^{p}ds)\frac{e^{-\mu(p-1) t}}{\mu^{\vartheta p+p-1}}) \label{equ 1  in proof lem hardy}.\end{eqnarray}
 Then (\ref{equ 1 in clm in proof lem hardy}) and (\ref{equ 1  in proof lem hardy})  yield 
 \begin{eqnarray*}
& & \int^{\infty}_{\frac{1}{10}}(e^{\mu t}t^{b}\int^{\infty}_{t}e^{-\mu s}f(s-t)^{\vartheta}ds)^{p}dt\leq \frac{1}{\mu^{p\vartheta+p-1}}(\int^{\infty}_{\frac{1}{10}}e^{-\mu s}f^{p}ds)(\int^{s}_{\frac{1}{10}}e^{\mu t}t^{pb}dt)
 \\&\leq & \frac{C_{l_{\mu},p,b}}{\mu^{p(1+\gamma)}}\int^{\infty}_{0}f^{p}s^{pb}ds.\ \ \ \ \textrm{The proof of  (\ref{equ Lem Hardy inequality exponential case}) is complete}.
 \end{eqnarray*}
 The proof of  (\ref{equ Lem Hardy inequality exponential case mu<}) is similar.\end{proof}
 \begin{proof}[\textbf{Proof of  Remark \ref{Rmk not Fredholm}}:]   We only have to show the "only if".  Without loss of generality, we only consider the Schauder theory, and   assume $m_{0}=1$, $\gamma=b=1,\ \overrightarrow{\beta}=0$, $\sigma_{1}=\sigma_{2}=Id$.  For any $\phi_{0}\in Ker B_{P^{0}}$,  let $f=\frac{\phi_{0}}{t}$. When $P=\frac{\partial }{\partial t}-B$ (translation-invariant on the end) and $t\geq 10$, the general solution to $Pu=\frac{\phi_{0}}{t}$ is $(\log t+C)\phi_{0}\notin \widehat{C}^{1,\alpha}_{0,0}$. Hence any extension $\widetilde{f} $  of $f$ to the whole $N$ is not in $ Range P|_{\widehat{C}^{1,\alpha}_{0,0}}$.
 
     On the other hand, let $f_{k}=\frac{\phi_{0}}{t^{1+\frac{1}{k}}}$,  Theorem \ref{Thm local invertibility}  implies  each $f_{k}$ admits an extension  $\widetilde{f}_{k}\in Range P|_{\widehat{C}^{1,\alpha}_{0,0}}$, and $\widetilde{f}_{k}$ tends to $\widetilde{f}\in C^{\alpha}_{0,1}$ which extends $f$. Then  $Range P|_{\widehat{C}^{1,\alpha}_{0,0}}$ is not closed in $C^{\alpha}_{0,1}$.\end{proof}


\small

\begin{thebibliography}{0}
\bibitem{AdamSimon}D. Adam, L.Simon. \emph{Rates of Asymptotic Convergence Near Isolated Singularities of Geometric Extrema}. Indiana Univ. Math. J. 37 (1988), no. 2, 225-254.
\bibitem{Amrouch} C. Amrouche, V. Girault, J. Giroire. \emph{Dirichlet and neumann exterior problems for the n-dimensional Laplace operator an approach in weighted Sobolev spaces}. Journal de Mathématiques Pures et Appliquées
Volume 76, Issue 1, January 1997, 55-81. 
\bibitem{APS}M.F. Atiyah, V.K. Patodi, I.M. Singer. \emph{Spectral asymmetry and Riemannian Geometry. I}. Math. Proc. Camb. Phil. Soc. 1975. 


\bibitem{Calabi}E. Calabi, \emph{Extremal K\"ahler metrics}. Seminar on Differential Geometry,
volume 16 of 102, 259–290. Ann. of Math. Studies, University Press, 1982.

\bibitem{ChangYang} S.Y. A. Chang, P. C. Yang. \emph{Non-linear Partial Differential Equations
in Conformal Geometry} ICM 2002. Vol. I.189-207.
\bibitem{DonaldsonKronheimer} S.K. Donaldson, P.B. Kronheimer. \emph{The geometry of Four-Manifolds}. Oxford Mathematical Monographs. 1990. 
 \bibitem{Evans}L.C. Evans. \emph{Partial differential equations}. Graduate Studies in Mathematics, Vol 19. AMS.

 \bibitem{GT} D. Gilbarg, N.S. Trudinger.  \emph{Elliptic Partial Differential Equations of Second Order}. Springer.
\bibitem{Hardy3}G.H. Hardy, J.E. Littlewood, G. P\'olya. \emph{Inequalities}. second edition, Cambridge, 1952. 
\bibitem{LockhartActa} R.B. Lockhart, R.C. McOwen. \emph{On elliptic systems in $\R^{n}$}. Acta Math. 150 (1983), 125-135.
\bibitem{Lockhart} R.B. Lockhart, R.C. McOwen. \emph{Elliptic differential operators on noncompact manifolds}. Annali della Scuola Normale Superiore di Pisa - Classe di Scienze (1985)
Volume 12, Issue 3, 409-447.
\bibitem{Mazzeo} R. Mazzeo. \emph{Elliptic theory of differential edge operators I}. Comm. Partial Differential Equations 16 (1991), no.10, 1615-1664.  
\bibitem{Mazya}V.G. Maz'ya, B.A. Plamenevskii. \emph{Estimates in $L_{p}$ and in H\"older Classes and the Miranda-Agmon Maximum Principle for Solutions of Elliptic Boundary Value Problems in Domains with Singular Points on the Boundary}. Translation of Math. Nachr. 81 (1978), 25-82.
\bibitem{MelroseMendoza}R.B. Melrose, G. Mendoza. \emph{Elliptic operators of totally characteristic type}. MSRI, Berkeley, CA June 1983. MSRI 047-83.
\bibitem{Myself2016bigpaper}Y.Q. Wang. \emph{Deformation of singular connections I: $G_{2}-$instantons with point singularities}. arXiv:1602.05701.
\bibitem{LSimon} L. Simon. \emph{Asymptotics for a class of non-linear evolution equations, with applications to geometric problems}. Annals of Mathematics, 118 (1983), 525-571. 
\bibitem{Yosida}K. Yosida. \emph{Functional analysis}. Springer Classics in Mathematics. Second edition.
\bibitem{Zhanggongqing}G.Q. Zhang, Y.Q. Lin. \emph{Lecture Notes on Functional Analysis}, Vol. 1 (Mandarin Chinese). Peking University Press; 1st edition. March 1, 1987.
\end{thebibliography}
\end{document}